\documentclass[12pt,a4paper]{amsart}
\calclayout
\numberwithin{equation}{section}
\allowdisplaybreaks
\theoremstyle{plain}

\newtheorem{Thm}{Theorem}[section]
\newtheorem{Lemma}[Thm]{Lemma}
\newtheorem{Coro}[Thm]{Corollary}

\newtheorem{Prop}[Thm]{Proposition}
\newtheorem{Rem}[Thm]{Remark}
\newtheorem{Ex}[Thm]{Example}

\usepackage{enumerate}
\usepackage{amsmath}
\usepackage{amssymb}
\usepackage{mathtools}
\mathtoolsset{showonlyrefs}
\usepackage{mathrsfs}
\usepackage{comment}
\usepackage{hyperref}
\usepackage{xcolor}
\usepackage{extarrows}
\usepackage{graphicx}
\usepackage{hyperref}
\usepackage{subfigure}
\usepackage{doi}
\usepackage{cases}
\usepackage{tikz}
\usetikzlibrary{positioning, shapes.geometric}

\def\i{\mathrm{i}}
\def\k{\mathbf{k}}
\def\j{\mathbf{j}}
\def\I{\mathcal{I}}
\def\A{\mathbf{A}}
\def\B{\mathbf{B}}
\def\D{\mathcal{D}}
\def\L{\mathcal{L}}

\begin{document}
	\title[]{Kernel representation formula from complex to real Wiener-It\^o integrals and vice versa}
	\author[]{Huiping Chen, Yong Chen and Yong Liu}
	\address{Huiping Chen\\LMAM, School of Mathematical Sciences\\Peking University\\ Beijing 100871\\ China}
	\email{chenhp@pku.edu.cn}
	\thanks{We thank Prof. Xiaohong Lan for valuable comments and discussion}
    \address{Yong Chen\\ School of Mathematics and Statistics\\
    	Jiangxi Normal University\\ Nanchang, Jiangxi 330022\\ China}
    \email{zhishi@pku.org.cn; chenyong77@gmail.com}
    \thanks{ Y. Chen is supported by NSFC (No. 11961033)}
     \address{Yong Liu\\ LMAM, School of Mathematical Sciences\\Peking University\\ Beijing 100871\\ China}
    \email{liuyong@math.pku.edu.cn}
    \thanks{  Y. Liu is supported by NSFC (No. 11731009, No. 11926327) and Center for Statistical Science, PKU}

	\begin{abstract}
	We clearly characterize the relation between real and complex Wiener-It\^o integrals. Given a complex multiple Wiener-It\^o integral, we get explicit expressions for two kernels of its real and imaginary parts. Conversely, consider a two-dimensional real Wiener-It\^o integral, we obtain the representation formula by a finite sum of complex Wiener-It\^o integrals. The main tools are a recursion technique and Malliavin derivative operators. We build a bridge between real and complex Wiener-It\^o integrals.
	\end{abstract}
	\maketitle

	
	\tableofcontents	
	\section*{Notations}

\begin{center}
	\begin{tabular}{rl}
		$\mathfrak{H}$&: a real separable Hilbert space\\
		$\mathfrak{H}\oplus\mathfrak{H}$&: the Hilbert space direct sum\\
		$\mathfrak{H}_{\mathbb{C}}$&: the complexification of $\mathfrak{H}$\\
		$\mathfrak{H}^{\odot m}$, $\left( \mathfrak{H}\oplus\mathfrak{H}\right) ^{\odot m}$, $\mathfrak{H}_{\mathbb{C}}^{\odot m}$&: the $m$ times symmetric tensor product of $\mathfrak{H}$, 	$\mathfrak{H}\oplus\mathfrak{H}$, $\mathfrak{H}_{\mathbb{C}}$\\
		$X$, $Y$&: the real Gaussian isonormal process over $\mathfrak{H}$\\
		$W$&: the real Gaussian isonormal process over $\mathfrak{H}\oplus\mathfrak{H}$\\
		$X_{\mathbb{C}}$, $Y_{\mathbb{C}}$&: the complexification of $X$, $Y$\\
		$Z$&: the complex Gaussian isonormal process over $\mathfrak{H}_{\mathbb{C}}$\\
		$\mathcal{H}_n(X)$, $\mathcal{H}_n(Y)$, $\mathcal{H}_n(W)$&: the $n$-th Wiener-It\^{o} chaos of $X$, $Y$, $W$\\
		$\mathcal{H}^{\mathbb{C}}_{n}(X)$, $\mathcal{H}^{\mathbb{C}}_{n}(W)$&: the complexification of $\mathcal{H}_n(X)$, $\mathcal{H}_n(W)$\\
		$\mathscr{H}_{m,n}(Z)$&: the $(m,n)$-th complex Wiener-It\^{o} chaos of $Z$\\
		symm($f\otimes g$), $f\tilde{\otimes} g$&: symmetric tensor product of $f$ and $g$\\
		$\Lambda$&: the set of all sequences $\textbf{p}=\left\lbrace p_k \right\rbrace_{k=1}^{\infty}$ of non-negative \\&\;\, integers with a finite sum\\
		$\mathrm{Re}\,z,\mathrm{Im}\,z$&: the real and imaginary parts of a complex number $z$\\
		$\left\lbrace e_k=e^1_k+\mathrm{i}e^2_k\right\rbrace_{k\geq1} $&: complete and orthogornal elements with norm $\sqrt2$ in $\mathfrak{H}_{\mathbb{C}},$\\ &\;\, $e^1_k=\text{Re}\,e_k$, $e^2_k=\text{Im}\,e_k$\\
		$\left\lbrace u_{1,0}(k), v_{1,0}(k)\right\rbrace_{k\geq1} $&: a complete and orthonormal basis of $\mathfrak{H}\oplus\mathfrak{H}$,\\ &\;\, $u_{1,0}(k)=\frac{1}{\sqrt2}(e^1_k,-e^2_k)$, $v_{1,0}(k)=\frac{1}{\sqrt2}(e^2_k, e^1_k)$
	\end{tabular}
\end{center}

\section{Introduction}\label{Section1 Introduction}

In 1951, It\^o published the seminal article \cite{ito1951real} and defined multiple Wiener integral with respect to a normal random measure which was introduced first and termed polynomial chaos by Wiener in \cite{Wiener1938}. It\^o showed that multiple Wiener integrals of different degrees are orthogonal to each other and closely related to Hermite polynomials. By making use of the relation between multiple Wiener integrals and Hermite polynomials, It\^o developed the chaos decomposition theory which leads to an orthogonal expansion of any square integrable functional of the normal random measure. Shortly afterwards, It\^o in \cite{ito1952complex} established the theory of complex multiple Wiener-It\^o integrals with respect to a complex normal random measure in 1952. It\^o firstly defined Hermite polynomials of complex variables, also first called Hermite-Laguerre-It\^o polynomials in \cite{ChenLiu2014}, and showed the close relation between Hermite polynomials of complex variables and complex multiple Wiener-It\^o integrals. Further, It\^o utilized the results obtained to generalize and derive the spectral structure and ergodicity of the shift transformation of normal screw lines such as complex Wiener processes and complex normal stationary processes. Thereafter, there has been a recent interest in theoretical studies on complex Gaussian fields, specifically complex Wiener-It\^o integrals, and Hermite-Laguerre-It\^o polynomials. For example, in \cite{campese2015fourth,ChenHuWang2017,chen2017fourth}, the authors showed the fourth moment theorem for complex multiple Wiener-It\^o integrals. That is, a sequence of complex multiple Wiener-It\^o integrals converges in distribution to the bivariate normal distribution if and only if the absolute moment up to fourth converges. The product formula, independence and asymptotic moment-independence for complex multiple Wiener-It\^o integrals were obtained in \cite{Chen2017}. In \cite{chenliu2019}, the author got some properties of complex multiple Wiener-It\^o integrals, complex Ornstein-Uhlenbeck operators and semigroups. For Hermite-Laguerre-It\^o polynomials, detailed analytic properties, relation with real Hermite polynomials and other related researches can be found in \cite{Agorram2016,Cotfas_2010,Ismail2016}.

The researches of complex Gaussian fields are strongly motivated by a variety of applications. For instance, in probabilistic models of cosmic microwave background radiation, it is important to understand high-frequency behaviour, namely central limit theorems for the Fourier coefficients associated with complex Gaussian subordinated fields which are identically distributed to functionals of complex Wiener-It\^o integrals, see \cite{Kamionkowski_1997,marinucci_peccati_2011} for more details. The cubic complex Ginzburg-Landau equation is one of the most crucial nonlinear  partial differential equations in applied mathematics and physics which describes various physical phenomena such as nonlinear waves, second-order phase transition, superconductivity and superfluidity, see \cite{RevModPhys.74.99}. The stochastic cubic complex Ginzburg-Landau equation with complex space-time white noise on the three dimensional torus can be studied by using theories of complex Wiener-It\^o integrals, see \cite{Hoshino2017} for instance. The Ornstein-Uhlenbeck process was introduced in \cite{arato1962evaluation} to model the Chandler wobble or variation of latitude concerning with the rotation of the earth, and later has been heavily used in finance and econophysics. Statistical inference for parameter estimators of the Ornstein-Uhlenbeck process such as consistency and asymptotic normality can be obtained by using techniques of complex Wiener-It\^o integrals, see \cite{ChenHuWang2017,ShenTangYin2021} for example. Moreover, the eigenfunctions of an Ornstein-Uhlenbeck operator which are closely related to the Hermite-Laguerre-It\^o polynomials have been found to be useful for applications such as simulating rare events and approximating solutions to the Fokker-Planck equation, see \cite{ChenLiu2014,A.Tantet2020,Zhang2021,ZHANG2022}. In the field of communication and signal processing, the noise is often supposed to be complex Gaussian noise, see \cite{Aghaei2008,BARONE2005,Matalkah2008,Reisenfeld2003} for more details. 

In order to investigate problems concerned with practical models mentioned above, we have to develop theories of complex	square integrable functionals of a complex Gaussian process and deeply understand the structure of complex Wiener-It\^o integrals based on It\^o's theories in \cite{ito1952complex}. We aim to not only explore specific properties of complex Wiener-It\^o integrals but also establish the relation between real and complex Wiener-It\^o integrals to make use of abundant theories of real Wiener-It\^o integrals in \cite{Malliavin1997,NNP2013,nourdin2012normal,nualart2006malliavin,NT2017} and so on. Chen and Liu in \cite[Theorem 3.3]{chen2017fourth} showed that the real and imaginary parts of a complex $(p,q)$-th Wiener-It\^o integral can be expressed as a real Wiener-It\^{o} integral of order $p+q$, respectively. This result provides a natural and direct perspective that one-dimensional complex Wiener-It\^o integral can be regarded as a two-dimensional real Wiener-It\^o integral. In this paper, we focus on completely characterizing the expressions for kernels of one-dimensional complex and two-dimensional real Wiener-It\^o integrals and thus build a bridge between them. One the one hand, the explicit and computable representations for kernels guarantee that we can solve theoretical and practical problems concerning with complex Wiener-It\^o integrals by utilizing the theories of two-dimensional real Wiener-It\^o integrals. For example, a sufficient and necessary condition for the existence of the density of a complex Wiener-It\^o integral is derived by using \cite[Theorem 3.1]{NNP2013} or \cite[Theorem 3]{NT2017}, see Corollary \ref{simple density exist}; also, in practical models, some statistical properties of a complex Wiener-It\^o integral can be easily verified with the help of explicit expressions for kernels, see Example \ref{ex} and Example \ref{inverse ex} for instance. On the other hand, we realize that, for a complex Wiener-It\^o integral of a higher order, the representations for kernels of real and imaginary parts are rather complicated and seems difficult to directly apply to some practical models. This reminds us the necessity to further develop the theories of complex Wiener-It\^o integral itself.

\cite[Theorem 3.3]{chen2017fourth} is not convenient to use sometimes since it only shows the existences of kernels, for which expressions depend on some redundant parameters, of real and imaginary parts of a complex Wiener-It\^o integral. In Section \ref{Section3}, we show how to get explicit representations for kernels of real and complex parts of a complex multiple Wiener-It\^o integral. In Section \ref{Section3.1}, we remove those redundant parameters in the proof of \cite[Theorem 3.3]{chen2017fourth} and prove the uniqueness of the kernels of the real and imaginary parts, see Theorem \ref{First proof}. In Section \ref{Section3.2}, based on a complete orthonormal system of complex Wiener-It\^{o} chaos shown by It\^o in \cite[Theorem 14]{ito1952complex} and defined as \eqref{ONB of (p,q)chaos}, we obtain more explicit recursion formulae by an induction argument, which makes full use of recursion formulae concerning complex multiple Wiener integrals established by It\^o in \cite[Theorem 9]{ito1952complex}, see \eqref{u_recurrence}, \eqref{v_recurrence} and Theorem \ref{Thm 1}. We stress that recursion formulae actually offer an algorithm to derive the representation for kernels of real and complex parts of a complex multiple Wiener-It\^o integral of a higher order. In Section \ref{Section3.4}, we combine the Stroock's formula with the relation between the real and complex Malliavin derivative operators, and then get another computable expressions for the kernels of the real and imaginary parts of a complex multiple Wiener-It\^o integral, see Corollary \ref{1 generalized Stroock}. In a word, we refine \cite[Theorem 3.3]{chen2017fourth} and clearly characterize the kernels of the real and imaginary parts of a complex multiple Wiener-It\^o integral.

In Section \ref{Section4}, we conversely consider a complex random variable whose real and imaginary parts are two real multiple Wiener-It\^o integrals. Essentially, we show how to represent a two-dimensional real Wiener-It\^o integrals as a finite sum of one-dimensional complex Wiener-It\^o integrals. We firstly prove the uniqueness of this representation in Section \ref{Section4.1}, see Theorem \ref{inverse First proof}. In Section \ref{Section4.2}, we get computable expressions for kernels of complex Wiener-It\^o integrals by using the complex Stroock's formula and the relation between the real and complex Malliavin derivative operators, see Theorem \ref{inverse expression of kernel}. Furthermore, in Theorem \ref{2 to 1}, we derive an equivalent condition that a two-dimensional real and a one-dimensional complex Wiener-It\^o integral can be represented by each other.

The paper is organized as follows. Section \ref{Section2 Preliminaries} introduces some elements of the real and complex Gaussian isonormal process and Malliavin calculus. In Section \ref{Section3.1}, Section \ref{Section3.2} and Section \ref{Section3.4}, we clearly characterize the kernels of real and imaginary parts of a complex multiple Wiener-It\^o integral. In section \ref{Section3.3}, we revisit the classical theory of It\^o's complex multiple integrals and realize the results in Section \ref{Section3.2} to complex multiple integrals with respect to a complex Brownian motion. In Section \ref{Section4.1} and Section \ref{Section4.2}, we prove that a complex random variable, whose real and imaginary parts are two real multiple Wiener-It\^o integrals, can be uniquely expressed as a finite sum of complex Wiener-It\^o integrals and expressions for kernels of these complex Wiener-It\^o integrals are obtained. The proofs of the main theorems of this paper are presented in Section \ref{Section5}.

\section{Preliminaries}\label{Section2 Preliminaries}

In this section, we briefly introduce some basic theories of the real isonormal Gaussian process, the complex isonormal Gaussian process and Malliavin calculus. See \cite{chen2017fourth,ito1952complex,nourdin2012normal,nualart2006malliavin} for more details.
\subsection{Real isonormal Gaussian process}
Suppose that $\mathfrak{H}$ is a real separable Hilbert space with an inner product denoted by $\left\langle \cdot, \cdot\right\rangle _\mathfrak{H}$. Let $\|h\|_{\mathfrak{H}}$ denote the norm of $h\in\mathfrak{H}$. Consider an isonormal Gaussian process $X=\left\{ X(h): h\in\mathfrak{H}\right\} $ defined on a complete probability space $(\Omega, \mathcal{F}, P)$, where the $\sigma$-algebra $\mathcal{F}$ is generated by $X$. That is, $X=\left\{ X(h): h\in\mathfrak{H}\right\} $ is a Gaussian family of centered random variables such that $\mathbb{E}\left[ X(h)X(g) \right]=\left\langle h,g\right\rangle _\mathfrak{H}$ for any $h,g\in \mathfrak{H}$. 

For $n\geq0$, the $n$-th Wiener-It\^{o} chaos $\mathcal{H}_n(X)$ of $X$ is the closed linear subspace of $L^2(\Omega)$ generated by the random variables $\left\{ H_n(X(h)): h\in\mathfrak{H}, \|h\|_{\mathfrak{H}}=1\right\} $, where $H_n(x)$ is the Hermite polynomial of degree $n$ defined by the equality
\begin{equation*}
	\exp\left\lbrace tx-\frac{1}{2}t^2\right\rbrace =\sum_{n=0}^{\infty}\frac{t^n}{n!}H_n(x).
\end{equation*}
We denote by $\Lambda$ the set of all sequences $\textbf{a}=\left\lbrace a_k \right\rbrace_{k=1}^{\infty}$ of non-negative integers with a finite sum. Set $\textbf{a}!=\prod_{k=1}^{\infty}a_k!$ and $|\textbf{a}|=\sum_{k=1}^{\infty}a_k$. For $\textbf{a}=\left\lbrace a_k \right\rbrace_{k=1}^{\infty}, \textbf{b}=\left\lbrace b_k \right\rbrace_{k=1}^{\infty}\in \Lambda$, we define $\textbf{a}+\textbf{b}=\left\lbrace a_k+ b_k \right\rbrace_{k=1}^{\infty}$ and say $\textbf{a}\leq \textbf{b}$ if $a_k\leq b_k$ for all $k\geq1$. Let $\mathfrak{H}^{\otimes n}$ and $\mathfrak{H}^{\odot n}$ denote the $n$-th tensor product and the $n$-th symmetric tensor product of $\mathfrak{H}$, respectively. Let $\left\lbrace\eta_k,k\geq1 \right\rbrace $ be a complete orthonormal system in $\mathfrak{H}$. For a sequence $\textbf{m}=\left\lbrace m_k \right\rbrace_{k=1}^{\infty}\in\Lambda $, define the Fourier-Hermite polynomial as
\begin{equation*}
	\textbf{H}_{\textbf{m}}=\frac{1}{\sqrt{\textbf{m}!}}\prod_{k=1}^{\infty}H_{m_k}\left(X\left(\eta_k\right)\right).
\end{equation*}
Denote by symm($f\otimes g$) the symmetrization of $f\otimes g$, where $f,g\in\mathfrak{H}$. Let $|\textbf{m}|=n$. For $n\geq1$, the mapping 
\begin{equation}\label{def of real integral}
	I_n\left(\rm{symm}\left(\otimes_{k=1}^{\infty}\eta_k^{\otimes m_k}\right)\right):=\sqrt{\textbf{m}!}\textbf{H}_{\textbf{m}}
\end{equation}
provides a linear isometry between the symmetric tensor product $\mathfrak{H}^{\odot n}$, equipped with the norm $\sqrt{n!}\|\cdot\|_{\mathfrak{H}^{\otimes n}}$, and the $n$-th Wiener-It\^{o} chaos $\mathcal{H}_{n}(X)$. For $n=0$, $I_0$ is the identity map. Note that It\^o firstly proved \eqref{def of real integral} in \cite[Theorem 3.1]{ito1951real}. For any $f\in \mathfrak{H}^{\odot n}$, the random variable $I_n(f)$ is called the real $n$-th Wiener-It\^o integral of $f$ with respect to $X$.	Wiener-It\^o chaos decomposition of $L^2(\Omega,\sigma(X),P)$ implies that $L^{2}(\Omega)$ can be decomposed into the infinite orthogonal sum of the spaces $\mathcal{H}_{n}(X)$. That is, any random variable $F \in L^2(\Omega,\sigma(X),P)$ admits a unique expansion of the form
\begin{equation*}
	F=\sum_{n=0}^{\infty} I_{n}\left(f_{n}\right),	
\end{equation*}
where $f_{0}=\mathbb{E}[F]$, and $f_{n} \in \mathfrak{H}^{\odot n}$ with $n \geq1$ are uniquely determined by $F$.

Given $f\in\mathfrak{H}^{\odot p}$, $g\in\mathfrak{H}^{\odot q}$, for $r=0,\dots,p\land q$, the $r$-th contraction of $f$ and $g$ is an element of $\mathfrak{H}^{\otimes (p+q-2r)}$ defined by 
\begin{equation*}
	f\otimes_rg=\sum_{i_1,\ldots,i_r=1}^{\infty}\left\langle f,\eta_{i_1}\otimes\cdots\otimes \eta_{i_r}\right\rangle _{\mathfrak{H}^{\otimes r}}\otimes\left\langle g,\eta_{i_1}\otimes\cdots\otimes \eta_{i_r}\right\rangle _{\mathfrak{H}^{\otimes r}}.
\end{equation*}
Notice that $f\otimes_rg$ is not necessarily symmetric, we denote by $f\tilde{\otimes}_rg$ its symmetrization. \cite[Proposition 2.7.10]{nourdin2012normal} provides the product formula for real  multiple Wiener-It\^{o} integrals as follows. For $f\in\mathfrak{H}^{\odot p}$ and $g\in\mathfrak{H}^{\odot q}$ with $p,q\geq0$, 
\begin{equation}\label{Product_formula}
	I_p(f)I_q(g)=\sum_{r=0}^{p\land q}r!\binom{p}{r}\binom{q}{r}I_{p+q-2r}(f\tilde{\otimes}_rg).
\end{equation}

\subsection{Complex isonormal Gaussian process}
Next, we introduce the complex isonormal Gaussian process. We complexify $\mathfrak{H}$, $L^2(\Omega)$ in the usual way and denote by $\mathfrak{H}_{\mathbb{C}}$, $L^2_{\mathbb{C}}(\Omega)$ respectively. Suppose $\mathfrak{h}=f+\i g\in\mathfrak{H}_{\mathbb{C}}$ with $f,g\in\mathfrak{H}$, we write $$X_{\mathbb{C}}(\mathfrak{h}):=X(f)+\i X(g), $$ which satisfies $\mathbb{E}\left[X_{\mathbb{C}}\left(\mathfrak{h}\right)\overline{X_{\mathbb{C}}\left(\mathfrak{h}'\right)}\right]=\left\langle \mathfrak{h},\mathfrak{h}'\right\rangle_{\mathfrak{H}_{\mathbb{C}}}$ with $\mathfrak{h}'\in\mathfrak{H}_{\mathbb{C}}$. Let $Y=\left\{ Y(h): h\in\mathfrak{H}\right\} $ is an independent copy of the isonormal Gaussian process $X$ over $\mathfrak{H}$. Define $Y_{\mathbb{C}}(\mathfrak{h})$ same as above.
Let
\begin{equation}\label{complex Gaussian process}
	Z(\mathfrak{h}):=\frac{X_{\mathbb{C}}(\mathfrak{h}) + \i Y_{\mathbb{C}}(\mathfrak{h})}{\sqrt{2}},\; \mathfrak{h}\in \mathfrak{H}_{\mathbb{C}},
\end{equation}
and we call $Z=\left\{ Z(\mathfrak{h}): \mathfrak{h}\in \mathfrak{H}_{\mathbb{C}}\right\} $  a complex isonormal Gaussian process over $\mathfrak{H}_{\mathbb{C}}$, which is a centered symmetric complex Gaussian family satisfying
\begin{equation}
	\mathbb{E}[Z(\mathfrak{h})^2]=0,\quad
	\mathbb{E}[Z(\mathfrak{h})\overline{Z(\mathfrak{h}')}]=\left\langle \mathfrak{h},\mathfrak{h}' \right\rangle_{\mathfrak{H}_{\mathbb{C}}}, \; \forall \mathfrak{h},\mathfrak{h}'\in \mathfrak{H}_{\mathbb{C}}.
\end{equation}

For each $p,q\geq 0$, let $\mathscr{H}_{p,q}(Z)$ be the closed linear subspace of $L^2_{\mathbb{C}}(\Omega)$ generated by the random variables $\left\{ J_{p,q}(Z(\mathfrak{h})): \mathfrak{h}\in\mathfrak{H}_{\mathbb{C}},\|\mathfrak{h}\|_{\mathfrak{H}_{\mathbb{C}}}=\sqrt2\right\} $, where $J_{p,q}(z)$ is the complex Hermite polynomial, or Hermite-Laguerre-It\^o polynomial, given by 
\begin{equation*}
	\exp\left\{ \lambda\bar{z}+\bar{\lambda}z-2|\lambda|^2\right\} =\sum_{p=0}^{\infty}\sum_{q=0}^{\infty}\frac{\bar{\lambda}^p\lambda^q}{p!q!}J_{p,q}(z),\;\lambda\in\mathbb{C}.
\end{equation*}
The space $\mathscr{H}_{p,q}(Z)$ is called the $(p, q)$-th Wiener-It\^o chaos of $Z$.

Take a complete orthonormal system $\left\{ \xi_k, k\geq1\right\} $ in $\mathfrak{H}_{\mathbb{C}}$. For two sequences $\textbf{p}=\left\{ p_k \right\}_{k=1}^{\infty}, \textbf{q}=\left\{ q_k \right\}_{k=1}^{\infty}\in\Lambda$, define a complex Fourier-Hermite polynomial or Fourier-Hermite-Laguerre-It\^o polynomial as
\begin{equation}\label{ONB of (p,q)chaos}
	\textbf{J}_{\textbf{p},\textbf{q}}:=\prod_{k=1}^{\infty}\frac{1}{\sqrt{2^{p_k+q_k}p_k!q_k!}}J_{p_k,q_k}\left(\sqrt2Z\left({\xi}_k\right)\right).
\end{equation}
Then for any $p, q \geq0$, the random variables
\begin{equation*}
	\left\{ \textbf{J}_{\textbf{p},\textbf{q}}:|\textbf{p}|=p, |\textbf{q}|=q\right\} ,
\end{equation*}
form a complete orthonormal system in $\mathscr{H}_{p,q}(Z)$. As a consequence, the linear mapping
\begin{equation}\label{def of complex integral}
	{\I}_{p,q}\left(\rm{symm}\left(\otimes_{k=1}^{\infty}{\xi}_k^{\otimes p_k}\right)\otimes \rm{symm}\left(\otimes_{k=1}^{\infty}\overline{\xi}_k^{\otimes q_k}\right)\right):=\sqrt{\textbf{p!q!}} \textbf{J}_{\textbf{p},\textbf{q}},
\end{equation}
provides an isometry from the tensor product $\mathfrak{H}_{\mathbb{C}}^{\odot p}\otimes\mathfrak{H}_{\mathbb{C}}^{\odot q}$, equipped with the norm\\ $\sqrt{p!q!}\|\cdot\|_{\mathfrak{H}_{\mathbb{C}}^{\otimes (p+q)}}$, onto the $(p,q)$-th Wiener-It\^o chaos $\mathscr{H}_{p,q}(Z)$.  Note that \eqref{def of complex integral} was proved by It\^o in \cite[Theorem 13.2]{ito1952complex}. For any $f\in\mathfrak{H}_{\mathbb{C}}^{\odot p}\otimes\mathfrak{H}_{\mathbb{C}}^{\odot q}$, $\I_{p,q}(f)$ is called complex $(p,q)$-th Wiener-It\^o integral of $f$ with respect to $Z$. Complex Wiener-It\^o chaos decomposition of $L_{\mathbb{C}}^2(\Omega,\sigma(Z),P)$ implies that $L_{\mathbb{C}}^2(\Omega,\sigma(Z),P)$ can be decomposed into the infinite orthogonal sum of the spaces $\mathscr{H}_{p,q}(Z)$. That is, any random variable $F \in L_{\mathbb{C}}^2(\Omega,\sigma(Z),P)$ admits a unique expansion of the form
\begin{equation}\label{complex chaos decomposition}
	F=\sum_{p=0}^{\infty}\sum_{q=0}^{\infty} \I_{p,q}\left(f_{p,q}\right),	
\end{equation}
where $f_{0,0}=\mathbb{E}[F]$, and $f_{p,q} \in \mathfrak{H}_{\mathbb{C}}^{\odot p}\otimes\mathfrak{H}_{\mathbb{C}}^{\odot q}$ with $p+q \geq1$, are uniquely determined by $F$.

Given $f\in\mathfrak{H}_{\mathbb{C}}^{\odot a}\otimes\mathfrak{H}_{\mathbb{C}}^{\odot b}$, $g\in\mathfrak{H}_{\mathbb{C}}^{\odot c}\otimes\mathfrak{H}_{\mathbb{C}}^{\odot d}$, for $i=0,\dots,a\land d$, $j=0,\dots,b\land c$, the $(i,j)$-th contraction of $f$ and $g$ is an element of $\mathfrak{H}_{\mathbb{C}}^{\odot (a+c-i-j)}\otimes\mathfrak{H}_{\mathbb{C}}^{\odot (b+d-i-j)}$ defined by
\begin{align}
	f \otimes_{i, j} g
	=& \sum_{l_{1}, \ldots, l_{i+j}=1}^{\infty}\left\langle f, \xi_{l_{1}} \otimes \cdots \otimes \xi_{l_{i}} \otimes \bar{\xi}_{l_{i+1}} \otimes \cdots \otimes \bar{\xi}_{l_{i+j}}\right\rangle\\ &\qquad\qquad\otimes\left\langle g, \xi_{l_{i+1}} \otimes \cdots \otimes \xi_{l_{i+j}} \otimes \bar{\xi}_{l_{1}} \otimes \cdots \otimes \bar{\xi}_{l_{i}}\right\rangle,
\end{align}
and by convention, $f \otimes_{0,0} g=f \otimes g$ denotes the tensor product of $f$ and $g$. \cite[Theorem 2.1]{Chen2017} and  \cite[Theorem A.1]{Hoshino2017} establish the product formula for complex Wiener-It\^o integrals. For $f \in \mathfrak{H}^{\odot a} \otimes\mathfrak{H}^{\odot b}$ and $ g \in \mathfrak{H}^{\odot c} \otimes \mathfrak{H}^{\odot d}$ with $a, b, c, d \geq0$,
\begin{equation}\label{complex product}
	\I_{a, b}(f) \I_{c, d}(g)=\sum_{i=0}^{a \wedge d} \sum_{j=0}^{b \wedge c}\binom{a}{i} \binom{d}{i}\binom{b}{j}\binom{c}{j} i! j! \I_{a+c-i-j, b+d-i-j}\left(f \otimes_{i, j} g\right).
\end{equation}

Here, we introduce some properties of real and complex polynomials that will be used in Section \ref{Section3.1}. Let $z=x+\i y$ with $x,y\in\mathbb{R}$. \cite[Corollary 2.8]{ChenLiu2014} shows that the real and complex Hermite polynomials satisfy 
\begin{equation}\label{real and complex Hermite}
	J_{p, q}(z)=\sum_{j=0}^{p+q} \mathrm{i}^{p+q-j} \sum_{r+s=j}\binom{p}{r}\binom{q}{s}(-1)^{q-s}H_j(x)H_{p+q-j}(y),
\end{equation}
and 
\begin{equation}\label{inverse real and complex Hermite}
	H_m(x)H_{n}(y)=\frac{\i ^n}{2^{m+n}}\sum_{j=0}^{m+n}\sum_{r+s=j}\binom{m}{r}\binom{n}{s}(-1)^sJ_{j, m+n-j}\left(z \right).
\end{equation}
According to \cite[Proposition 3.11]{Hu2017} or \cite{Stroock1983}, for $\theta\in[0,2\pi)$, the real Hermite polynomials satisfy the invariant property
\begin{equation}\label{Hermite invariant}
	H_{n}\left(x \cos\theta+y\sin\theta\right)=\sum_{j=0}^{n}\binom{n}{j}\left( \cos\theta\right)^j \left(\sin\theta \right)^{n-j} H_j(x)H_{n-j}(y).
\end{equation}

\subsection{Malliavin derivative operators}\label{Section2.3}
Finally, we show some fundamental elements of Malliavin calculus. Let $\mathcal{S}$ denote the class of smooth random variables of the form $F=f(X(h_1),\dots,X(h_n))$, where $h_1,\dots,h_n\in\mathfrak{H}$, $n\geq1$ and $f\in C_p^{\infty}(\mathbb{R}^n)$, the set of all infinitely continuously differentiable real-valued functions such that all its partial derivatives have polynomial growth. Given $F\in \mathcal{S}$, the Malliavin derivative $DF$ is a $\mathfrak{H}$-valued random element given by 
\begin{equation*}
	DF=\sum_{i=1}^{n}\frac{\partial f}{\partial x_i}\left(X\left(h_1\right),\dots,X\left(h_n\right)\right)h_i.
\end{equation*}
The derivative operator $D$ is a closable and unbounded operator from $L^p(\Omega)$ to $L^p(\Omega;\mathfrak{H})$ for any $p\geq1$. By iteration, for $k\geq2$, one can define $k$-th derivative $D^kF\in L^p(\Omega;\mathfrak{H}^{\otimes k})$. For any $p\geq1$ and $k\geq0$, let $\mathbb{D}^{k,p}$ denote the closure of $\mathcal{S}$ with respect to the norm $\|\cdot\|_{k,p}$ given by 
\begin{equation*}
	\|F\|_{k,p}^{p}=\sum_{i=0}^{k}\mathbb{E}\left(\left\|D^iF\right\|^{p}_{\mathfrak{H}^{\otimes i}}\right).
\end{equation*} 
For any $p\geq1$ and $k\geq0$, we set $\mathbb{D}^{\infty,p}=\bigcap_{k\geq0}\mathbb{D}^{k,p}$, $\mathbb{D}^{k,\infty}=\bigcap_{p\geq1}\mathbb{D}^{k,p}$ and $\mathbb{D}^{\infty}=\bigcap_{k\geq0}\mathbb{D}^{k,\infty}$. If $F=I_p(f)$ with $f\in\mathfrak{H}^{\odot p}$, then $I_p(f)\in\mathbb{D}^{\infty}$ and for any $k\geq0$,
\begin{equation*}
	D^kI_p(f) =
	\begin{cases}
		\frac{p!}{(p-k)!}I_{p-k}(f),  & {k\leq p,}\\
		0,  & {k>p.}
	\end{cases}
\end{equation*} 	
In \cite{Stroock1987}, the Stroock's formula was established that for any random variable $F \in L^2(\Omega,\sigma(X),P)$ expanded as $	F=\mathbb{E}[F]+\sum_{n=1}^{\infty} I_{n}\left(f_{n}\right)$, if $F\in \mathbb{D}^{n,2} $ for some $n\geq1$, then 
\begin{equation}\label{real Stroock}
	f_p=\frac{1}{p!}\mathbb{E}\left[ D^pF\right] 
\end{equation}
for all $p\leq n$.

In \cite{chenliu2019}, Chen and Liu defined the complex Malliavin derivative operators $\D$ and $\bar{\D}$ as follows. Let $\mathcal{S}_Z$ denote the set of all smooth random variables of the form
\begin{equation}\label{complex smooth r.v.}
	G=g\left(Z\left(\mathfrak{h}_{1}\right), \cdots, Z\left(\mathfrak{h}_{m}\right)\right),
\end{equation}
where $\mathfrak{h}_1,\dots,\mathfrak{h}_m\in\mathfrak{H}_{\mathbb{C}}$, $m\geq1$ and $g\in C_p^{\infty}(\mathbb{C}^m)$, the set of all infinitely continuously differentiable complex-valued functions such that all its partial derivatives have polynomial growth.
If $G \in \mathcal{S}_Z$ with the form \eqref{complex smooth r.v.}, then the complex Malliavin derivatives of $G$ are the elements of $L_{\mathbb{C}}^{2}(\Omega; \mathfrak{H}_{\mathbb{C}})$ defined by
\begin{align}\label{definition of complex derivative}
	\D G &=\sum_{j=1}^{m} \partial_{j} g\left(Z\left(\mathfrak{h}_{1}\right), \cdots, Z\left(\mathfrak{h}_{m}\right)\right) \mathfrak{h}_{j}, \\
	\bar{\D} G &=\sum_{j=1}^{m} \bar{\partial}_{j} g\left(Z\left(\mathfrak{h}_{1}\right), \cdots, Z\left(\mathfrak{h}_{m}\right)\right) \overline{\mathfrak{h}}_{j},
\end{align}
where for complex numbers $z_j=x_j+\i y_j$ with $x_j,y_j \in\mathbb{R}$ and $j=1,\ldots,m$,
\begin{align}
	\partial_{j} g&=\frac{\partial}{\partial z_{j}} g\left(z_{1}, \ldots, z_{m}\right)=\frac{1}{ 2}\left( \frac{\partial}{\partial x_j}-\i \frac{\partial}{\partial y_j}\right)g\left(x_{1}, y_1,\ldots, x_{m}, y_m\right),\\
	\bar{\partial}_{j} g&=\frac{\partial}{\partial \bar{z}_{j}} g\left(z_{1}, \ldots, z_{m}\right)=\frac{1}{ 2}\left( \frac{\partial}{\partial x_j}+\i \frac{\partial}{\partial y_j}\right)g\left(x_{1}, y_1,\ldots, x_{m}, y_m\right),
\end{align}
are the Wirtinger derivatives.
One can define the iteration of the operator $\D$ and $\bar{\D}$ in such a way that $\D^{p} \bar{\D}^{q} G$ is a random variable with values in $\mathfrak{H}_{\mathbb{C}}^{\odot p} \otimes \mathfrak{H}_{\mathbb{C}}^{\odot q}$ for any $G \in \mathcal{S}_Z$. For $p+q\geq1$, $\D^{p} \bar{\D}^{q}$ are closable from $L_{\mathbb{C}}^{r}(\Omega)$ to $L_{\mathbb{C}}^{r}\left(\Omega, \mathfrak{H}_{\mathbb{C}}^{\odot p} \otimes \mathfrak{H}_{\mathbb{C}}^{\odot q}\right)$ for every $r \geq 1$. Denote by $\mathscr{D}^{p, r} \cap\bar{\mathscr{D}}^{q, r}$ the closure of $\mathcal{S}_Z$ with respect to the Sobolev seminorm $\|\cdot\|_{p, q, r}$ given by 
\begin{equation*}
	\|G\|_{p,q,r}^{r}=\sum_{i=0}^{p}\sum_{j=0}^{q}\mathbb{E}\left(\left\|\D^i\bar{\D}^jG\right\|^{r}_{\mathfrak{H}_{\mathbb{C}}^{\otimes (i+j)}}\right).
\end{equation*}
In \cite[Theorem 2.8]{chenliu2019}, the complex Stroock's formula was showed that for every random variable $G\in L_{\mathbb{C}}^{2}(\Omega, \sigma(Z), P)$ expressed as $G=\sum_{p=0}^{\infty} \sum_{q=0}^{\infty} I_{p, q}\left(g_{p, q}\right)$, with $g_{p, q} \in \mathfrak{H}_{\mathbb{C}}^{\odot p} \otimes \mathfrak{H}_{\mathbb{C}}^{\odot q},
$
if $G \in \mathscr{D}^{m, 2} \cap \bar{\mathscr{D}}^{n, 2}$, then
\begin{equation}\label{complex stroock}
	g_{p, q}=\frac{1}{p ! q !} \mathbb{E}\left[\D^{p} \bar{\D}^{q} G\right], \quad \forall \;p \leq m, q \leq n .
\end{equation}

\section{Kernel representation formula from complex to real Wiener-It\^o integrals}\label{Section3}

We firstly define the real isonormal Gaussian process $W$ over $\mathfrak{H}\oplus\mathfrak{H}$. Let $h, f\in\mathfrak{H}$, denote by $(h,f)\in\mathfrak{H}\oplus\mathfrak{H}$ the Cartesian product of $\mathfrak{H}$ and $\mathfrak{H}$. With respect to the natural inner product, that is, for any $h_1, h_2, f_1, f_2\in\mathfrak{H}$,
\begin{equation*}
	\left\langle (h_1, f_1),(h_2, f_2)\right\rangle_{\mathfrak{H}\oplus\mathfrak{H}}=\left\langle h_1,h_2\right\rangle _{\mathfrak{H}}+\left\langle f_1,f_2\right\rangle _{\mathfrak{H}},
\end{equation*}
$\mathfrak{H}\oplus\mathfrak{H}$ is a Hilbert space. \cite[Proposition 4.2]{chen2017fourth} offers a realization of the isonormal Gaussian process $W$ over the Hilbert space $\mathfrak{H}\oplus\mathfrak{H}$ as
\begin{equation}\label{Gaussian W}
	W = \left\{ W(h, f) = X(h) + Y (f): h, f \in\mathfrak{H}\right\},
\end{equation} 
where $X$ and $Y$ are two real independent identically distributed isonormal Gaussian processes over $\mathfrak{H}$. We denote $\mathcal{H}_n(W)$ as the $n$-th Wiener-It\^o chaos of $W$.

The following lemma shows that we can obtain a complete orthonormal basis of $\mathfrak{H}\oplus\mathfrak{H}$ by using the complete and orthogonal elements $\left\{ e_k\right\}_{k\geq1} $ in $\mathfrak{H}_{\mathbb{C}}$ with $\left\|e_k \right\|_{\mathfrak{H}_{\mathbb{C}}} =\sqrt2$.
\begin{Lemma}\label{lemma_basis}
	Let $\left\{ e_k=e_k^1+\mathrm{i}e_k^2\right\}_{k\geq1} $ be complete and orthogonal in $\mathfrak{H}_{\mathbb{C}}$. Suppose $\left\|e_k \right\|^2_{\mathfrak{H}_{\mathbb{C}}}=\left\|e_k^1 \right\|^2_{\mathfrak{H}}+\left\|e_k^2 \right\|^2_{\mathfrak{H}}=2$. Define $u_{1,0}(k)=\frac{1}{\sqrt2}\left(e_k^1,-e_k^2\right)$ and $v_{1,0}(k)=\frac{1}{\sqrt2}\left(e_k^2,e_k^1\right)$. Then $\left\{ u_{1,0}(k),v_{1,0}(k)\right\}_{k\geq1} $ is a complete orthonormal basis of $\mathfrak{H}\oplus\mathfrak{H}$.
\end{Lemma}

The proof of Lemma \ref{lemma_basis} is presented in Section \ref{Section5.1}.

\begin{Rem}\label{expand z}
	Actually, the definition of complex isonormal Gaussian process $Z$ over $\mathfrak{H}_{\mathbb{C}}$, see \eqref{complex Gaussian process}, leads us to construct the basis of $\mathfrak{H}\oplus\mathfrak{H}$ as above in Lemma \ref{lemma_basis}. Specifically, calculating $Z(e_k)$ directly, we get
	\begin{align}
		Z(e_k)&=\frac{X_{\mathbb{C}}(e_k)+\mathrm{i}Y_{\mathbb{C}}(e_k)}{\sqrt{2}}
		=\frac{X(e_k^1)+\mathrm{i}X(e_k^2)+\mathrm{i}\left(Y(e_k^1)+\mathrm{i}Y(e_k^2)\right)}{\sqrt{2}}
		\\&=\frac{1}{\sqrt2}[X\left(e_k^1\right)-Y(e_k^2)]+\frac{\mathrm{i}}{\sqrt2}[X(e_k^2)+Y(e_k^1)]
		\\&=W\left(\frac{1}{\sqrt2}(e_k^1,-e_k^2)\right)+\i W\left(\frac{1}{\sqrt2}(e_k^2,e_k^1)\right)
		\\&=I_1\left(\frac{1}{\sqrt2}(e_k^1,-e_k^2)\right)+\i I_1\left(\frac{1}{\sqrt2}(e_k^2,e_k^1)\right)
		\\&=I_1\left(u_{1,0}(k)\right) +\i I_1\left(v_{1,0}(k)\right),
	\end{align}
	where $I_{1}(\cdot)$ is the real $1$-th  Wiener-It\^{o} integral with respect to $W$. Then,
	\begin{equation}\label{I_10}
		\I_{1,0}(e_k)=J_{1,0}\left(Z(e_k)\right)
		=Z(e_k)=I_1\left(u_{1,0}(k)\right) +\i I_1\left(v_{1,0}(k)\right).
	\end{equation}
	This is why we construct $\left\{ u_{1,0}(k),v_{1,0}(k)\right\}_{k\geq1} $ as the orthonormal basis of $\mathfrak{H}\oplus\mathfrak{H}$.
\end{Rem}

\subsection{Uniqueness theorem}\label{Section3.1} 

Let	$\textbf{p}=\left\lbrace p_k \right\rbrace_{k=1}^{\infty}$, $\textbf{q}=\left\lbrace q_k\right\rbrace_{k=1}^{\infty}\in\Lambda$ with $|\textbf{p}|=p$ and $|\textbf{q}|=q$, respectively. For simplicity of presentation, define $$a_{k,j}:=\mathrm{i}^{p_k+q_k-j} \sum_{r+s=j}\binom{p_k}{r}\binom{q_k}{s}(-1)^{q_k-s}.$$
For $\textbf{j}=\left\lbrace j_k \right\rbrace_{k=1}^{\infty}\in \Lambda$ with $\textbf{j}\leq \textbf{p}+\textbf{q}$, $\prod_{k=1}^{\infty}a_{k,j_k}$ is a complex number. Define $u(\textbf{p},\textbf{q}), v(\textbf{p},\textbf{q})\in(\mathfrak{H}\oplus\mathfrak{H})^{\odot (p+q)} $ as 
\begin{align}
	u(\textbf{p},\textbf{q})&=\sum_{\textbf{j}\leq \textbf{p}+\textbf{q}} \mathrm{Re}\left( \prod_{l=1}^{\infty}a_{l,j_l}\right) \mathrm{symm}\left( \otimes_{k=1}^{\infty} \left( u_{1,0}(k)^{\otimes j_k}\otimes v_{1,0}(k)^{\otimes (p_k+q_k-j_k)} \right)\right) ,\\
	v(\textbf{p},\textbf{q})&=\sum_{\textbf{j}\leq \textbf{p}+\textbf{q}}\mathrm{Im}\left( \prod_{l=1}^{\infty}a_{l,j_l}\right)\mathrm{symm}\left( \otimes_{k=1}^{\infty} \left( u_{1,0}(k)^{\otimes j_k}\otimes v_{1,0}(k)^{\otimes (p_k+q_k-j_k)} \right)\right) .
\end{align}

Let $\left\lbrace e_k\right\rbrace_{k\geq1}$ be as in the statement of Lemma \ref{lemma_basis}, that is, $\left\{ e_k\right\}_{k\geq1} $ are complete and orthogonal with $\left\|e_k \right\|_{\mathfrak{H}_{\mathbb{C}}}=\sqrt2$ in $\mathfrak{H}_{\mathbb{C}}$. For $f\in\mathfrak{H}_{\mathbb{C}}^{\odot p}\otimes\mathfrak{H}_{\mathbb{C}}^{\odot q}$, there exists a unique sequence $$\left\lbrace c(\textbf{p};\textbf{q}):\textbf{p},\textbf{q}\in\Lambda,|\textbf{p}|=p, |\textbf{q}|=q \right\rbrace \subseteq\mathbb{C}
$$ satisfying $\sum_{\textbf{p},\textbf{q}\in\Lambda,|\textbf{p}|=p, |\textbf{q}|=q}\left| c(\textbf{p};\textbf{q})\right| ^2<\infty$ such that 
\begin{equation}\label{another expansion}
	f=\sum_{\textbf{p},\textbf{q}\in\Lambda,|\textbf{p}|=p, |\textbf{q}|=q} c(\textbf{p};\textbf{q})\mathrm{symm}\left( \otimes_{k=1}^{\infty}e_k^{\otimes p_k}\right) \otimes  \mathrm{symm}\left( \otimes_{k=1}^{\infty}\overline{e}_k^{\otimes q_k}\right).
\end{equation}
Now, we state the uniqueness theorem as follows.
\begin{Thm}[Uniqueness Theorem]\label{First proof}
	Suppose $f\in\mathfrak{H}_{\mathbb{C}}^{\odot p}\otimes\mathfrak{H}_{\mathbb{C}}^{\odot q}$ with the expansion given by \eqref{another expansion}.	Then $\I_{p,q}(f)$ admits the unique representation
	\begin{align}
		\I_{p,q}(f)
		&=\sum_{\textbf{p},\textbf{q}\in\Lambda,|\textbf{p}|=p, |\textbf{q}|=q}  I_{p+q}\left(\mathrm{Re}\left( c(\textbf{p};\textbf{q})\right) u(\textbf{p},\textbf{q})-\mathrm{Im}\left( c(\textbf{p};\textbf{q})\right) v(\textbf{p},\textbf{q})\right)\\& \quad+\i \sum_{\textbf{p},\textbf{q}\in\Lambda,|\textbf{p}|=p, |\textbf{q}|=q} I_{p+q}\left( \mathrm{Im}\left( c(\textbf{p};\textbf{q})\right) u(\textbf{p},\textbf{q})+\mathrm{Re}\left( c(\textbf{p};\textbf{q})\right) v(\textbf{p},\textbf{q})\right) ,
	\end{align} 
	where $I_{p+q}(\cdot)$ is the real $(p+q)$-th Wiener-It\^{o} integral with respect to $W$.  
\end{Thm}

Note that, 
\begin{align}
	&\left\lbrace 2^{-\frac{p+q}{2}}e_k^{\otimes p}\otimes\overline{e}_k^{\otimes q}:k\geq 1\right\rbrace\\\subseteq& \left\lbrace \frac{\left( \otimes_{k=1}^{\infty}e_k^{\otimes p_k}\right)  \otimes   \left( \otimes_{k=1}^{\infty}\overline{e}_k^{\otimes q_k}\right)}{2^{\frac{p+q}{2}}}:  \textbf{p}=\left\lbrace p_k \right\rbrace_{k=1}^{\infty},\textbf{q}=\left\lbrace q_k \right\rbrace_{k=1}^{\infty}\in\Lambda, |\textbf{p}|=p,|\textbf{q}|=q\right\rbrace , 
\end{align}
and the latter is a complete orthonormal basis of $\mathfrak{H}_{\mathbb{C}}^{\odot p}\otimes\mathfrak{H}_{\mathbb{C}}^{\odot q}$. Therefore, the proof of uniqueness theorem (Theorem \ref{First proof}) is divided into following three steps:
\begin{itemize}
	\item Unique representation for $\I_{p, q}\left(e_k^{\otimes p}\otimes\overline{e}_k^{\otimes q}\right) $, namely, Lemma \ref{lemma repre for sub basis}.
	\item Unique representation for $	\I_{p,q}\left(\mathrm{symm}\left( \otimes_{k=1}^{\infty}e_k^{\otimes p_k}\right) \otimes  \mathrm{symm}\left( \otimes_{k=1}^{\infty}\overline{e}_k^{\otimes q_k}\right)\right) $, namely, Proposition \ref{Prop repre for basis}.
	\item Unique representation for $	\I_{p,q}\left(f\right)$ with $f\in\mathfrak{H}_{\mathbb{C}}^{\odot p}\otimes\mathfrak{H}_{\mathbb{C}}^{\odot q}$.
\end{itemize}

\begin{Lemma}\label{lemma repre for sub basis}
	For $ e_1=e_1^1+\mathrm{i}e_1^2\in \mathfrak{H}_{\mathbb{C}}$ with $\left\|e_1 \right\|_{\mathfrak{H}_{\mathbb{C}}}=\sqrt2$,
	\begin{align}\label{repre for sub basis}
		&\I_{p, q}\left(e_1^{\otimes p}\otimes\overline{e}_1^{\otimes q}\right)
		\\=&\sum_{j=0}^{p+q} \mathrm{i}^{p+q-j} \sum_{r+s=j}\binom{p}{r}\binom{q}{s}(-1)^{q-s}I_{p+q}\left( u_{1,0}(1)^{\otimes j}\otimes v_{1,0}(1)^{\otimes (p+q-j)}\right),
	\end{align}
	where $$u_{1,0}(1)=\frac{1}{\sqrt2}\left(e_1^1,-e_1^2\right), \quad v_{1,0}(1)=\frac{1}{\sqrt2}\left(e_1^2,e_1^1\right),$$ and $I_{p+q}(\cdot)$ is the real  $(p+q)$-th Wiener-It\^{o} integral with respect to $W$. 
\end{Lemma}

\begin{Prop}\label{Prop repre for basis}
	For	$\textbf{p}=\left\lbrace p_k \right\rbrace_{k=1}^{\infty}, \textbf{q}=\left\lbrace q_k \right\rbrace_{k=1}^{\infty} \in\Lambda$ with $|\textbf{p}|=p$ and $|\textbf{q}|=q$,
	\begin{equation}
		\I_{p,q}\left(\mathrm{symm}\left( \otimes_{k=1}^{\infty}e_k^{\otimes p_k}\right) \otimes  \mathrm{symm}\left( \otimes_{k=1}^{\infty}\overline{e}_k^{\otimes q_k}\right)\right)=I_{p+q}\left(u(\textbf{p},\textbf{q})\right)  +\i I_{p+q}\left(v(\textbf{p},\textbf{q})\right) ,
	\end{equation}
	where $I_{p+q}(\cdot)$ is the real $(p+q)$-th Wiener-It\^{o} integral with respect to $W$.  	
\end{Prop}

Proofs of Lemma \ref{lemma repre for sub basis}, Proposition \ref{Prop repre for basis} and Theorem \ref{First proof} are presented in Section \ref{Section5.2}.

\begin{Rem}\label{remark}
	Chen and Liu in \cite[Theorem 3.3]{chen2017fourth} showed that for $f\in\mathfrak{H}_{\mathbb{C}}^{\odot p}\otimes\mathfrak{H}_{\mathbb{C}}^{\odot q}$, there exist $u,v\in\left(  \mathfrak{H}\oplus\mathfrak{H}\right) ^{\odot (p+q)}$ such that 
	\begin{equation}
		\mathrm{Re}\,\I_{p,q}(f)=I_{p+q}(u),\quad\mathrm{Im}\,\I_{p,q}(f)=I_{p+q}(v),
	\end{equation}
	where $I_{p+q}(\cdot)$ is the real $(p+q)$-th Wiener-It\^o integral with respect to $W$. In order to prove \cite[Theorem 3.3]{chen2017fourth}, Chen and Liu introduced $n+1$ parameters $0<\theta_{n}<\dots<\theta_{0}<\pi$ and defined the invertible matrix  $\mathrm{M}$ as
	\begin{align}
		\mathrm{M}&=\mathrm{M}\left(\theta_{0}, \ldots, \theta_{n}\right)=\left( \mathrm{M}_{ij}\right)_{0\leq i,j\leq n} \\
		&=\left[\begin{array}{ccccc}
			\left( \sin\theta_{0}\right)^n & \binom{n}{1}\left( \sin\theta_{0}\right)^{n-1}\cos\theta_{0}&\cdots&\binom{n}{n-1}\sin\theta_{0}\left( \cos\theta_{0}\right)^{n-1}&	\left( \cos\theta_{0}\right)^n \\
			\left( \sin\theta_{1}\right)^n & \binom{n}{1}\left( \sin\theta_{1}\right)^{n-1}\cos\theta_{1}&\cdots&\binom{n}{n-1}\sin\theta_{1}\left( \cos\theta_{1}\right)^{n-1}&	\left( \cos\theta_{1}\right)^n \\
			\vdots&\vdots&\vdots&\vdots&\vdots\\
			\left( \sin\theta_{n}\right)^n & \binom{n}{1}\left( \sin\theta_{n}\right)^{n-1}\cos\theta_{n}&\cdots&\binom{n}{n-1}\sin\theta_{n}\left( \cos\theta_{n}\right)^{n-1}& \left( \cos\theta_{n}\right)^n 
		\end{array}\right].
	\end{align}
	They used the properties of Hermite polynomials \eqref{real and complex Hermite} and \eqref{Hermite invariant}, solved $n+1$ linear equations and expressed $J_{k, n-k}\left(e_1^{\otimes k}\otimes\overline{e}_1^{\otimes (n-k)}\right)$ (see \cite[Equation (21)]{chen2017fourth}) as
	\begin{align}\label{21}
		J_{k, n-k}\left(Z(e_1)\right)
		&=\sum_{j=0}^{n} c_jH_j\left( \mathrm{Re} Z(e_1)\right) H_{n-j}\left( \mathrm{Im} Z(e_1)\right)\\
		&=\sum_{j=0}^{n} c_j\sum_{l=0}^{n}\left( \mathrm{M}^{-1}\right) _{jl}H_n\left( W(\cos\theta_lu_{1,0}(1)+\sin\theta_lv_{1,0}(1))\right),
	\end{align}
	where $c_j=\mathrm{i}^{n-j} \sum_{r+s=j}\binom{k}{r}\binom{n-k}{s}(-1)^{n-k-s}$.
	
	Actually, the right hand side of \eqref{21} does not depend on these redundant parameters $\theta_0,\ldots,\theta_{n}$. From the perspective on Wiener-It\^o integrals, utilizing the linearity, we get that 
	\begin{align}\label{repre for symmetric tensor product}
		&\sum_{l=0}^{n}\left( \mathrm{M}^{-1}\right) _{jl}H_n\left( W(\cos\theta_lu_{1,0}(1)+\sin\theta_lv_{1,0}(1))\right)\\
		=&\sum_{l=0}^{n}\left( \mathrm{M}^{-1}\right) _{jl}I_n\left(\left(  \cos\theta_lu_{1,0}(1)+\sin\theta_lv_{1,0}(1)\right) ^{\otimes n}\right) \\
		=&\sum_{l=0}^{n}\left( \mathrm{M}^{-1}\right) _{jl}I_n\left( \sum_{r=0}^{n}\binom{n}{r}\left( \cos\theta_l\right) ^r \left( \sin\theta_l\right) ^{n-r} u_{1,0}(1)^{\otimes r}\otimes v_{1,0}(1)^{\otimes (n-r)}\right) \\
		=&\sum_{r=0}^{n}\sum_{l=0}^{n}\left( \mathrm{M}^{-1}\right) _{jl} \mathrm{M} _{lr}I_n\left( u_{1,0}(1)^{\otimes r}\otimes v_{1,0}(1)^{\otimes (n-r)}\right)\\
		=&\sum_{r=0}^{n}\left( \mathrm{M}^{-1}\mathrm{M}\right) _{jr} I_n\left( u_{1,0}(1)^{\otimes r} v_{1,0}(1)^{\otimes (n-r)}\right)\\
		=&I_n\left( u_{1,0}(1)^{\otimes j}\otimes v_{1,0}(1)^{\otimes (n-j)}\right).
	\end{align}
	That is, 
	\begin{align}\label{in simple proof}
		&\I_{k, n-k}\left(e_1^{\otimes k}\otimes\overline{e}_1^{\otimes (n-k)}\right)=J_{k, n-k}\left(Z\left( e_1\right) \right)\\
		=&\sum_{j=0}^{n} \mathrm{i}^{n-j} \sum_{r+s=j}\binom{k}{r}\binom{n-k}{s}(-1)^{n-k-s}I_n\left( u_{1,0}(1)^{\otimes j}\otimes v_{1,0}(1)^{\otimes (n-j)}\right),
	\end{align}
	which is exactly what we state in Lemma \ref{lemma repre for sub basis} and implies that the right hand side of \eqref{21} does not depend on these redundant parameters $\theta_0,\ldots,\theta_{n}$. As a by-product, we get an expression for  $\mathrm{symm}\left( u_{1,0}(1)^{\otimes j}\otimes v_{1,0}(1)^{\otimes (n-j)}\right)$ as
	\begin{align}
		&\mathrm{symm}\left( u_{1,0}(1)^{\otimes j}\otimes v_{1,0}(1)^{\otimes (n-j)}\right)\\=&\sum_{l=0}^{n}\left( \mathrm{M}^{-1}\right) _{jl}\left(  \cos\theta_lu_{1,0}(1)+\sin\theta_lv_{1,0}(1)\right) ^{\otimes n},\quad \forall 0<\theta_{n}<\dots<\theta_{0}<\pi.
	\end{align}
	
\end{Rem}

\subsection{Representation theorem}\label{Section3.2}

Note that $$\left\lbrace 2^{-\frac{p+q}{2}}\left( \otimes_{k=1}^{\infty}e_k^{\otimes p_k}\right)  \otimes   \left( \otimes_{k=1}^{\infty}\overline{e}_k^{\otimes q_k}\right)  :\textbf{p}=\left\lbrace p_k \right\rbrace_{k=1}^{\infty},\textbf{q}=\left\lbrace q_k \right\rbrace_{k=1}^{\infty}\in\Lambda, |\textbf{p}|=p,|\textbf{q}|=q \right\rbrace$$ is a complete orthonormal basis of $\mathfrak{H}_{\mathbb{C}}^{\odot p}\otimes\mathfrak{H}_{\mathbb{C}}^{\odot q}$. In Section \ref{Section3.1}, based on the representation for $\I_{p, q}\left(e_k^{\otimes p}\otimes\overline{e}_k^{\otimes q}\right) $ with $k\geq1$ (see Lemma \ref{lemma repre for sub basis}), we prove the uniqueness theorem (Theorem \ref{First proof}). This leads to somewhat complicated representation for $$	\I_{p,q}\left(\mathrm{symm}\left( \otimes_{k=1}^{\infty}e_k^{\otimes p_k}\right) \otimes  \mathrm{symm}\left( \otimes_{k=1}^{\infty}\overline{e}_k^{\otimes q_k}\right)\right) .$$ In this section, we directly get more explicit recursion  formulae for kernels of the real and imaginary parts of $	\I_{p,q}\left(\mathrm{symm}\left( \otimes_{k=1}^{\infty}e_k^{\otimes p_k}\right) \otimes  \mathrm{symm}\left( \otimes_{k=1}^{\infty}\overline{e}_k^{\otimes q_k}\right)\right) $ by an induction argument.

Before illustrating the representation theorem, we introduce some notations. Define $\k = (k_1,k_2,\ldots)$ and $\j = (j_1,j_2,\ldots)$. 
For ease of notations, we write
\begin{equation}
	\begin{aligned}
		\I_{p,q}(\k,\j) :&= \I_{p,q}(\mathrm{symm}(e_{k_1}\otimes\cdots\otimes e_{k_p})\otimes \mathrm{symm}(\overline{e_{j_1}}\otimes\cdots \otimes\overline{e_{j_q}})) \\&= \I_{p,q}(e_{k_1}\otimes\cdots\otimes e_{k_p}\otimes \overline{e_{j_1}}\otimes\cdots \otimes\overline{e_{j_q}}),
	\end{aligned}
\end{equation}	
and
\begin{align}
	u_{p,q}(\k,\j) &:= u_{p,q}(k_1,\dots, k_p;j_1,\dots, j_q),\\
	v_{p,q}(\k,\j) &:= v_{p,q}(k_1,\dots, k_p;j_1,\dots, j_q),
\end{align} 
where $u_{p,q}(k_1,\dots, k_p;j_1,\dots, j_q),\;v_{p,q}(k_1,\dots, k_p;j_1,\dots, j_q)\in(\mathfrak{H}\oplus\mathfrak{H})^{\odot (p+q)}$ are recursively defined by 
\begin{align}
	&u_{0,1}(j)=u_{1,0}(j), \quad
	v_{0,1}(j)=-v_{1,0}(j)\label{uv_10},\\
	u_{p,q}(\k,\j) &=u_{p-1,q}(\k,\j)\tilde{\otimes}u_{1,0}(k_p)- v_{p-1,q}(\k,\j) \tilde{\otimes}v_{1,0}(k_p) \label{u_recurrence} \\
	&=u_{p,q-1}(\k,\j)\tilde{\otimes}u_{0,1}(j_q)-v_{p,q-1}(\k,\j)\tilde{\otimes}v_{0,1}(j_q), \\
	v_{p,q}(\k,\j)&=u_{p-1,q}(\k,\j)\tilde{\otimes}v_{1,0}(k_p)\ + v_{p-1,q}(\k,\j)\tilde{\otimes}u_{1,0}(k_p) \label{v_recurrence} \\
	&=u_{p,q-1}(\k,\j)\tilde{\otimes}v_{0,1}(j_q)+v_{p,q-1}(\k,\j)\tilde{\otimes}u_{0,1}(j_q). 
\end{align}
Now we explain that $u_{p,q}(\k,\j)$ and $v_{p,q}(\k,\j)$ are well defined. Since 
$\I_{q,p}(\j,\k)=\overline{\I_{p,q}(\k,\j)}$,
we have
\begin{align}\label{uv_conjugate}
	u_{p,q}(\k,\j)=u_{q,p}(\j,\k) \quad \mbox{and} \quad
	v_{p,q}(\k,\j)&=-v_{q,p}(\j,\k).
\end{align}
Then
\begin{align}
	u_{p,q}(\k,\j)=u_{q,p}(\j,\k)
	&=u_{q-1,p}(\j,\k)\tilde{\otimes}u_{1,0}(j_q)-v_{q-1,p}(\j,\k)\tilde{\otimes}v_{1,0}(j_q)\\
	&=u_{p,q-1}(\k,\j)\tilde{\otimes}u_{0,1}(j_q)-v_{p,q-1}(\k,\j)\tilde{\otimes}v_{0,1}(j_q),
\end{align}
where the last equality follows from \eqref{uv_conjugate}. Thus the second equality of \eqref{u_recurrence} holds, that is, $u_{p,q}(\k,\j)$ is well defined. Similarly, one can show that $v_{p,q}(\k,\j)$ is well defined.

\begin{Thm}\label{prop_represent}
	For $p\geq0$, $q\geq0$, $p+q>0$, we have
	\begin{align}\label{I_represent}
		\I_{p,q}\left(\k,\j\right) = I_{p+q}\left(u_{p,q}\left(\k,\j\right)\right) +\mathrm{i}I_{p+q}\left( v_{p,q}\left(\k,\j\right)\right),
	\end{align}
	where $I_{p+q}(\cdot)$ is the real $(p+q)$-th Wiener-It\^{o} integral with respect to $W$.
\end{Thm}

\begin{Rem}
	Take $\k=(\underbrace{1,\ldots,1}_{p}, 0,\ldots)$, $\j=(\underbrace{1,\ldots,1}_{q}, 0,\ldots)$ and consider $\I_{p,q}(e_1^{\otimes p}\otimes\overline{e}_1^{\otimes q})$. Using the recursion formulae of $u_{p,q}(\k,\j)$ and $v_{p,q}(\k,\j)$ repeatly, we further get  
	\begin{equation}\label{kernel}
		u_{p,q}(\k,\j)+\i v_{p,q}(\k,\j) =\mathrm{symm}\left(\sum_{j=0}^p\sum_{l=0}^{q-1} \binom{p}{j}\binom{q-1}{l}\left( f_{p,q}(j,l)+\i g_{p,q}(j,l) \right) \right),
	\end{equation}
	where, after a tedious calculation, $f_{p,q}(j,l)$ and $g_{p,q}(j,l)$ are obtained as 	
	\begin{align}\label{expand}
		&f_{p,q}(j,l)+\i g_{p,q}(j,l)\\=&
		\i^{j+l}(-1)^lu_{1,0}(1)^{\otimes (p+q-(j+l+1))}\otimes v_{1,0}(1)^{\otimes (j+l)}\otimes\left(  u_{1,0}(1)- \i v_{1,0}(1)\right).
	\end{align}
	One can verify that \eqref{kernel} is equal to the kernel expression in \eqref{repre for sub basis} in the sense of symmetry.
\end{Rem}

Let $\left\{ e_k\right\}_{k\geq1} $ with $\left\|e_k \right\|_{\mathfrak{H}_{\mathbb{C}}} =\sqrt2$ be complete and orthogonal elements in $\mathfrak{H}_{\mathbb{C}}$. For $f\in\mathfrak{H}_{\mathbb{C}}^{\odot p}\otimes\mathfrak{H}_{\mathbb{C}}^{\odot q}$, there exists a unique sequence
$$\left\{ b_{p,q}(\k,\j):=b(k_1,\dots,k_p;j_1,\dots,j_q),k_1,\dots,k_p,j_1,\dots,j_q\geq1\right\}\subseteq{\mathbb{C}}$$ 
such that $\sum_{k_1,\dots,k_p,j_1,\dots,j_q=1}^{\infty}|b_{p,q}(\k,\j)|^2<\infty$ and 
\begin{equation}\label{complex kernel}
	f=\sum_{k_1,\dots,k_p,j_1,\dots,j_q=1}^{\infty}b_{p,q}(\k,\j)(\mathrm{symm}(e_{k_1}\otimes\cdots\otimes e_{k_p})\\
	\otimes\mathrm{symm}(\overline{e_{j_1}}\otimes\cdots\otimes \overline{e_{j_q}})).
\end{equation}

Now we are able to restate the uniqueness theorem (Theorem \ref{First proof}) as a version of recursion representation.	
\begin{Thm}[Representation Theorem]\label{Thm 1}
	Suppose $f\in\mathfrak{H}_{\mathbb{C}}^{\odot p}\otimes\mathfrak{H}_{\mathbb{C}}^{\odot q}$ with the expansion given by \eqref{complex kernel}. Then $\I_{p,q}(f)$ admits the unique representation
	\begin{align}
		\I_{p,q}(f) &=  \sum_{k_1,\dots,k_p,j_1,\dots,j_q=1}^{\infty} I_{p+q}\left(\mathrm{Re}\left( b_{p,q}(\k,\j)\right)  u_{p,q}\left(\k,\j\right)-\mathrm{Im}\left( b_{p,q}(\k,\j)\right)  v_{p,q}\left(\k,\j\right) \right)\\&\quad+\i \sum_{k_1,\dots,k_p,j_1,\dots,j_q=1}^{\infty} I_{p+q}\left(\mathrm{Im}\left( b_{p,q}(\k,\j)\right)  u_{p,q}\left(\k,\j\right)+\mathrm{Re}\left( b_{p,q}(\k,\j)\right)  v_{p,q}\left(\k,\j\right)\right),
	\end{align}  
	where $I_{p+q}(\cdot)$ is the real $(p+q)$-th Wiener-It\^{o} integral with respect to $W$.  
\end{Thm}	
We prove Theorem \ref{prop_represent} by induction and Theorem \ref{Thm 1} by an approximation argument in Section \ref{Section5.3}.

From now on, we assume that $D$ is the real Malliavin derivative operator with respect to the isonormal Gaussian process $W$ over $\mathfrak{H}\oplus\mathfrak{H}$. As a corollary of Theorem \ref{prop_represent}, we obtain a sufficient and necessary condition of the existence of density of the complex Wiener-It\^o integral $\I_{p,q}(\k,\j)$.

\begin{Coro}\label{simple density exist}
	Consider a complex random variable
	\begin{align}
		\I_{p,q}(e_{k_1}\otimes\cdots\otimes e_{k_p}\otimes \overline{e_{j_1}}\otimes\cdots \otimes\overline{e_{j_q}})  := F_1+\i F_2.
	\end{align}
	Without loss of generality, we assume that $
	1\leq k_1\leq\ldots\leq k_p, \, 1\leq j_1\leq\ldots\leq j_q$. Then the law of the two-dimensional random vector $(F_1,F_2)$ is absolutely continuous with respect to Lebesgue measure on $\mathbb{R}^2$ if and only if $p\neq q$ or if $p=q$ and there exists $1\leq l\leq p$ such that $k_l\neq j_l$.
\end{Coro}

\begin{proof}
	From Theorem \ref{prop_represent},
	\begin{equation*}
		F_1=I_{p+q}(u_{p,q}(\k,\j)),\quad F_2=I_{p+q}(v_{p,q}(\k,\j)).
	\end{equation*}
	By \eqref{u_recurrence} and the definition of symmetric tensor product,
	\begin{align}\label{symm u}
		u_{p,q}(\k,\j)=&\frac{1}{p+q}\Bigg[\sum_{l=1}^{p}(u_{p-1,q}(\widehat{\k}_l,\j)\otimes u_{1,0}(k_l) - v_{p-1,q}(\widehat{\k}_l,\j)\otimes v_{1,0}(k_l))
		\\&\qquad\quad+\sum_{r=1}^{q}(u_{p,q-1}(\k,\widehat{\j}_r) \otimes u_{0,1}(j_r) - v_{p,q-1}(\k,\widehat{\j}_r)\otimes v_{0,1}(j_r))\Bigg],
	\end{align}
	where the notations of $u_{p-1,q}(\widehat{\k}_l,\j)$, $v_{p-1,q}(\widehat{\k}_l,\j)$, $u_{p,q-1}(\k,\widehat{\j}_r)$ and $v_{p,q-1}(\k,\widehat{\j}_r)$ can be found in \eqref{kernel-1}. 
	Then,
	\begin{align}\label{DF1}
		&DF_1\\= & I_{p+q-1}\left(\sum_{l=1}^{p}u_{p-1,q}(\widehat{\k}_l,\j)\right)u_{1,0}(k_l) - I_{p+q-1} \left(\sum_{l=1}^{p}v_{p-1,q}(\widehat{\k}_l,\j)\right) v_{1,0}(k_l)  \\
		&+ I_{p+q-1}\left(\sum_{r=1}^{q}u_{p,q-1}(\k,\widehat{\j}_r)\right) u_{0,1}(j_r) - I_{p+q-1}\left(\sum_{r=1}^{q}v_{p,q-1}(\k,\widehat{\j}_r)\right) v_{0,1}(j_r)\\
		=&\sum_{l=1}^{p}I_{p+q-1}(u_{p-1,q}(\widehat{\k}_l,\j))u_{1,0}(k_l) - \sum_{l=1}^{p}I_{p+q-1}(v_{p-1,q}(\widehat{\k}_l,\j)) v_{1,0}(k_l)\\
		&\quad+\sum_{r=1}^{q}I_{p+q-1}(u_{p,q-1}(\k,\widehat{\j}_r)) u_{1,0}(j_r) +\sum_{r=1}^{q}I_{p+q-1}(v_{p,q-1}(\k,\widehat{\j}_r)) v_{1,0}(j_r).
	\end{align}
	Similarly,
	\begin{align}\label{DF2}
		&DF_2\\=&\sum_{l=1}^{p} I_{p+q-1}\left(v_{p-1,q}(\widehat{\k}_l,\j)\right)u_{1,0}(k_l) + \sum_{l=1}^{p}I_{p+q-1}\left(u_{p-1,q}(\widehat{\k}_l,\j)\right) v_{1,0}(k_l)\\
		&+\sum_{r=1}^{q}I_{p+q-1}\left(v_{p,q-1}(\k,\widehat{\j}_r)\right) u_{1,0}(j_r) - \sum_{r=1}^{q}I_{p+q-1}\left(u_{p,q-1}(\k,\widehat{\j}_r)\right) v_{1,0}(j_r).
	\end{align}
	The law of $(F_1,F_2)$ is absolutely continuous with respect to Lebesgue measure if and only if $DF_1(\omega)$ and $DF_2(\omega)$ are linearly independent in $\mathfrak{H}\oplus\mathfrak{H}$ for a.s. $\omega$, see \cite[Theorem 3.1]{NNP2013}. (The absolute continuity of the law is also equivalent to the linear independence of $u_{p,q}(\k,\j)$ and $v_{p,q}(\k,\j)$ in $\left( \mathfrak{H}\oplus\mathfrak{H}\right) ^{\odot (p+q)}$ by \cite[Theorem 3]{NT2017}. However, it is easier to verify the linearly independence in $\mathfrak{H}\oplus\mathfrak{H}$ than in $\left( \mathfrak{H}\oplus\mathfrak{H}\right) ^{\odot (p+q)}$.) Therefore, it suffices to show that $DF_1(\omega)$ and $DF_2(\omega)$ are linearly independent in $\mathfrak{H}\oplus\mathfrak{H}$ for a.s. $\omega$ if and only if $p\neq q$ or if $p=q$ and there exists $1\leq l\leq p$ such that $k_l\neq j_l$.

	Necessity: Otherwise, we have $p=q$ and for any $1\leq l\leq p$, $k_l=j_l$. Fix $\omega$, 
	\begin{align}
		DF_2(\omega) =& \sum_{l=1}^{p}I_{2p-1}\left(v_{p-1,p}(\widehat{\k}_l,\k)+v_{p,p-1}(\k,\widehat{\k}_l)\right)u_{1,0}(k_l)\\
		&+\sum_{l=1}^{p}I_{2p-1}\left(u_{p-1,p}(\widehat{\k}_l,\k)-u_{p,p-1}(\k,\widehat{\k}_l)\right) v_{1,0}(k_l)\\
		=&0,
	\end{align}
	which contradicts the assumption that for a.s. $\omega$, $DF_1(\omega)$ and $DF_2(\omega)$ are linearly independent in $\mathfrak{H}\oplus\mathfrak{H}$.
	
	Sufficiency: Observing 
	whether the coefficients of orthogonal elements $$\left\{ u_{1,0}(k_l),u_{1,0}(j_r),v_{1,0}(k_l),v_{1,0}(j_r)\right\} _{1\leq l\leq p, 1\leq r\leq q}\subseteq\mathfrak{H}\oplus\mathfrak{H}$$ in \eqref{DF1} and \eqref{DF2} can cancel out, we have that $DF_1(\omega)\neq0$ for a.s. $\omega$. Moreover, $DF_2(\omega)\neq0$ for a.s. $\omega$ if $p\neq q$ or if $p=q$ and there exists $1\leq l\leq p$ such that $k_l\neq j_l$. Fix $\omega$, combining the orthogonality of $\left\{ u_{1,0}(k_l),u_{1,0}(j_r),v_{1,0}(k_l),v_{1,0}(j_r)\right\} _{1\leq l\leq p, 1\leq r\leq q}$ and its  coefficient vectors are orthogonal in $\mathbb{R}^{2(p+q)}$, we get the conclusion.
	
\end{proof}

\subsection{Back to It\^o's theory}\label{Section3.3}

Assume that the underlying Hilbert space $\mathfrak{H}=L^{2}(T, \mathcal{B}, \mu)$, where $(T, \mathcal{B})$ is a measurable space and $\mu$ is a $\sigma$-finite measure without atoms. Then the linear mapping $\I_{p, q}$ coincides with a multiple Wiener-It\^o integrals defined by It\^o in \cite{ito1952complex}. According to \cite[Theorem 3.1]{ito1952complex}, there exists a continuous complex normal random measure $\mathbf{M}=\{M(B): B \in \mathcal{B}, \mu(B)<\infty\}$ on $(T, \mathcal{B})$, such that, for every $B, C \in \mathcal{B}$ with finite measure,
\begin{equation}
	\mathbb{E}[M(B) \overline{M(C)}]=\mu(B \cap C).
\end{equation}
For the off-diagonal simple function $f \in \mathfrak{H}_{\mathbb{C}}^{\otimes p} \otimes \mathfrak{H}_{\mathbb{C}}^{\otimes q}$ of the form
\begin{equation}
	f\left(t_{1}, \ldots, t_{p}, s_{1}, \ldots, s_{q}\right)=\sum a_{i_{1} \cdots i_{p} j_{1} \cdots j_{q}} \mathbf{1}_{B_{i_{1}} \times \cdots \times B_{i_{p}} \times B_{j_{1}} \times \cdots \times B_{j_{q}}},
\end{equation}
where $\mathbf{1}_{A}$ is the indicator function of a set $A$,
It\^o defined the multiple integral $\I_{p, q}(f)$ as
\begin{equation}
	\I_{p, q}(f)=\sum a_{i_{1} \cdots i_{p} j_{1} \cdots j_{q}} M\left(B_{i_{1}}\right) \cdots M\left(B_{i_{p}}\right) \overline{M\left(B_{j_{1}}\right)} \cdots \overline{M\left(B_{j_{q}}\right)}.
\end{equation}
Then by density argument, It\^o extended the multiple integrals to any $f \in \mathfrak{H}_{\mathbb{C}}^{\otimes p} \otimes$ $\mathfrak{H}_{\mathbb{C}}^{\otimes q}$ as follows,
\begin{equation}
	\I_{p, q}(f)=\int \cdots \int f\left(t_{1}, \ldots, t_{p}; s_{1} \ldots, s_{q}\right) \mathrm{d} M\left(t_{1}\right) \cdots \mathrm{d} M\left(t_{p}\right) \overline{\mathrm{d} M\left(s_{1}\right)} \cdots \overline{\mathrm{d} M\left(s_{q}\right)}.	
\end{equation}

Specifically, let $\mathfrak{H}=L^2(\mathbb{R}^{+})$ and  $\left(B_{1}(t), B_{2}(t)\right)_{t\geq0}$ be a two-dimensional Brownian motion. $\zeta_{t}:=\frac{B_{1}(t)+\mathrm{i} B_{2}(t)}{\sqrt{2}}$ is called complex Brownian motion. We can extend $  (\mathbf{1}_{[0,t_1]},\mathbf{1}_{[0,t_2]})\mapsto B_{1}(t_1)+B_{2}(t_2)  $ with $t_1,t_2\geq0$ to a real isonormal Gaussian process denoted by $\widehat{M}$ over $\mathfrak{H}\oplus\mathfrak{H}$, namely the stochastic integral with respect to two-dimensional Brownian motion.

For $0\leq r\leq p+q$, by the definition of It\^o's iterated integral, we know that 
\begin{align}\label{iterated}
	I_{d_{p+q-r}}
	:=&I_{p+q} \left( u_{1,0}(1)^{\otimes (p+q-r)}\otimes v_{1,0}(1)^{\otimes r}\right)\\
	=&\frac{(p+q)!}{2^{(p+q)/2}} \sum_{i_{1}, \ldots, i_{p+q}=1}^{2} (-1)^{d_{p+q-r} }\int_{0}^{\infty} \int_{0}^{t_{p+q}} \cdots \int_{0}^{t_{2}} e_1^{i_1}  \left(t_{1}\right)\\&
	\cdots e_1^{i_{p+q-r}}\left(t_{p+q-r}\right) e_1^{i_{p+q-r+1}}\left(t_{p+q-r+1}\right) \cdots e_{1}^{i_{p+q}}\left(t_{p+q}\right)\mathrm{d} B_{i_{1}}\left(t_{1}\right) \\&
	\cdots \mathrm{d} B_{i_{p+q-r}}\left(t_{p+q-r}\right) \mathrm{d} B_{3-i_{p+q-r+1}}\left(t_{p+q-r+1}\right)  \cdots \mathrm{d} B_{3-i_{p+q}}\left(t_{p+q}\right),
\end{align}	
where $I_{p+q}(\cdot)$ is $(p+q)$-th Wiener-It\^o integral with respect to $\widehat{M}$ and $d_{p+q-r}$ denotes the number of $j\in\left\lbrace 1,\ldots, p+q-r \right\rbrace $ such that $i_j=2$ and $d_{0}=0$ for $r=p+q$ .

Then using \eqref{kernel}, \eqref{expand} and \eqref{iterated}, we can express $\I_{p, q}\left(e_1^{\otimes p}\otimes\overline{e}_1^{\otimes q}\right)$ as It\^o's iterated integral
\begin{align}
	\I_{p, q}\left(e_1^{\otimes p}\otimes\overline{e}_1^{\otimes q}\right)&=\sum_{j=0}^p\sum_{l=0}^{q-1} \binom{p}{j}\binom{q-1}{l}\left( I_{p+q} \left( f_{p,q}(j,l)\right)+\i I_{p+q} \left(g_{p,q}(j,l)\right)\right)  \\
	&=\sum_{j=0}^p\sum_{l=0}^{q-1} \binom{p}{j}\binom{q-1}{l}\i^{j+l}(-1)^{l}	( I_{d_{p+q-(j+l)}}-\i I_{d_{p+q-(j+l+1)}}).
\end{align}
For example, take $p=q=1$, we get that
\begin{align}
	&\I_{1, 1}\left(e_1\otimes\overline{e}_1\right)=I_{d_{2}}+ I_{d_{0}} +\i \left( -I_{d_{1}}+I_{d_{1}}\right)\\
	=&\frac{2!}{2} \sum_{i_{1}, i_{2}=1}^{2} (-1)^{d_{2} }\int_{0}^{\infty} \int_{0}^{t_{2}} 
	e_1^{i_1}\left(t_{1}\right)  e_{1}^{i_2}\left(t_{2}\right)
	\mathrm{d} B_{i_{1}}\left(t_{1}\right) \cdots \mathrm{d} B_{i_{2}}\left(t_{2}\right)\\
	&+\frac{2!}{2} \sum_{i_{1}, i_{2}=1}^{2} (-1)^{0 }\int_{0}^{\infty} \int_{0}^{t_{2}} 
	e_1^{i_1}\left(t_{1}\right)  e_{1}^{i_2}\left(t_{2}\right)
	\mathrm{d} B_{3-i_{1}}\left(t_{1}\right)  \mathrm{d} B_{3-i_{2}}\left(t_{2}\right)\\
	=& \int_{0}^{\infty} \int_{0}^{t_{2}} 
	e_1^{1}\left(t_{1}\right)  e_{1}^{1}\left(t_{2}\right)
	\mathrm{d} B_{1}\left(t_{1}\right) \mathrm{d} B_{1}\left(t_{2}\right)+\int_{0}^{\infty} \int_{0}^{t_{2}} 
	e_1^{2}\left(t_{1}\right)  e_{1}^{2}\left(t_{2}\right)
	\mathrm{d} B_{2}\left(t_{1}\right) \mathrm{d} B_{2}\left(t_{2}\right)\\
	& -\int_{0}^{\infty} \int_{0}^{t_{2}} 
	e_1^{1}\left(t_{1}\right)  e_{1}^{2}\left(t_{2}\right)
	\mathrm{d} B_{1}\left(t_{1}\right) \mathrm{d} B_{2}\left(t_{2}\right)- \int_{0}^{\infty} \int_{0}^{t_{2}} 
	e_1^{2}\left(t_{1}\right)  e_{1}^{1}\left(t_{2}\right)
	\mathrm{d} B_{2}\left(t_{1}\right) \mathrm{d} B_{1}\left(t_{2}\right)\\
	&+ \int_{0}^{\infty} \int_{0}^{t_{2}} 
	e_1^{1}\left(t_{1}\right)  e_{1}^{1}\left(t_{2}\right)
	\mathrm{d} B_{2}\left(t_{1}\right) \mathrm{d} B_{2}\left(t_{2}\right)+\int_{0}^{\infty} \int_{0}^{t_{2}} 
	e_1^{2}\left(t_{1}\right)  e_{1}^{2}\left(t_{2}\right)
	\mathrm{d} B_{1}\left(t_{1}\right) \mathrm{d} B_{1}\left(t_{2}\right)\\
	&+\int_{0}^{\infty} \int_{0}^{t_{2}} 
	e_1^{1}\left(t_{1}\right)  e_{1}^{2}\left(t_{2}\right)
	\mathrm{d} B_{2}\left(t_{1}\right) \mathrm{d} B_{1}\left(t_{2}\right)+ \int_{0}^{\infty} \int_{0}^{t_{2}} 
	e_1^{2}\left(t_{1}\right)  e_{1}^{1}\left(t_{2}\right)
	\mathrm{d} B_{1}\left(t_{1}\right) \mathrm{d} B_{2}\left(t_{2}\right).
\end{align}

\subsection{Generalized Stroock's formula}\label{Section3.4}	

See \eqref{Gaussian W} for the definition of the isonormal Gaussian process $W$ over $\mathfrak{H}\oplus\mathfrak{H}$. The definitions of complex Malliavin derivative operators $\D$ and $\bar{\D}$, real Malliavin derivative operator $D$ with respect to $W$ are introduced in Section \ref{Section2.3}. Note that $DF\in L^p(\Omega;\mathfrak{H}\oplus\mathfrak{H})$ for $F\in\mathbb{D}^{1,p}$, we denote $DF$ by $\left( D_1F,D_2F\right) $. In particular, for \begin{equation}\label{real smooth r.v.}
	F=f(W(h_1,f_1),\dots,W(h_n,f_n)),
\end{equation}
where $(h_1,f_1),\dots,(h_n,f_n)\in\mathfrak{H}\oplus\mathfrak{H}$, $n\geq1$ and $f\in C_p^{\infty}(\mathbb{R}^n)$, $D_1F$ and $D_2F$ are $\mathfrak{H}$-valued random elements given by 
\begin{align}
	D_1F&=\sum_{i=1}^{n}\frac{\partial f}{\partial x_i}\left(W(h_1,f_1),\dots,W(h_n,f_n)\right)h_i,\\
	D_2F&=\sum_{i=1}^{n}\frac{\partial f}{\partial x_i}\left(W(h_1,f_1),\dots,W(h_n,f_n)\right)f_i.
\end{align}
For any operator $(A_1,A_2)$ and $(A_3,A_4)$, we define the tensor product of $(A_1,A_2)$ and $(A_3,A_4)$ as
\begin{equation}
	\begin{pmatrix}
		A_1,A_2
	\end{pmatrix}
	\otimes 
	\begin{pmatrix}
		A_3,A_4
	\end{pmatrix}=
	\begin{pmatrix}
		A_1\\A_2
	\end{pmatrix}
	\otimes 
	\begin{pmatrix}
		A_3\\A_4
	\end{pmatrix}=
	\left( A_1A_3,A_2A_3,A_1A_4,A_2A_4 \right)  ^T,
\end{equation}
where $(\cdot)^T$ denotes the transposition of a matrix or a vector. Note that for $k\geq2$ and $F\in\mathbb{D}^{k,p}$, $D^kF=\left( D_1,D_2\right)^{\otimes k}F\in L^p(\Omega;\left( \mathfrak{H}\oplus\mathfrak{H}\right) ^{\otimes k})$.

The following Lemma establishes the relation between the real Malliavin derivative operator $D=(D_1,D_2)$ and the complex Malliavin derivative operators $\D$, $\bar{\D}$.
\begin{Lemma}\label{Lemma relation r and c derivative}
	\begin{equation}\label{relation r and c derivative}
		\D=\frac{D_1-\i D_2}{\sqrt2 },\quad
		\bar{\D}= \frac{D_1+\i D_2}{\sqrt2 }.
	\end{equation}
\end{Lemma}

\begin{proof}
	By Remark \ref{expand z}, for a smooth random variable $G$ with the form \eqref{complex smooth r.v.}, $G$ can be expressed as 
	\begin{align}
		G&=g\left(Z\left(\mathfrak{h}_{1}\right), \cdots, Z\left(\mathfrak{h}_{m}\right)\right)\\
		&=g_1\left(\frac{W\left(h_1^1, -h_1^2  \right)}{\sqrt2}, \frac{W\left( h_1^2,h_1^1 \right)}{\sqrt2},  \cdots, \frac{W\left( h_m^1, -h_m^2 \right)}{\sqrt2}, \frac{W\left( h_m^2,h_m^1 \right)}{\sqrt2}\right)\\&\quad+\i g_2\left(\frac{W\left(h_1^1, -h_1^2  \right)}{\sqrt2}, \frac{W\left( h_1^2,h_1^1 \right)}{\sqrt2},  \cdots, \frac{W\left( h_m^1, -h_m^2 \right)}{\sqrt2}, \frac{W\left( h_m^2,h_m^1 \right)}{\sqrt2}\right),
	\end{align}
	where $\mathfrak{h}_j=h_j^1+\i h_j^2\in\mathfrak{H}_{\mathbb{C}}$ with $1\leq j \leq m$ and $g_1, g_2\in C_p^{\infty}(\mathbb{R}^{2m})$. Then by the definitions of the complex Malliavin derivative operators $\D$, $\bar{\D}$ and real Malliavin derivative operator $D=(D_1,D_2)$, we get the conclusion.
\end{proof}

\begin{Rem}
	\eqref{relation r and c derivative} is entirely analogous to the definition of Wirtinger derivatives
	\begin{equation}
		\frac{\partial}{\partial z}=\frac{1}{ 2}\left( \frac{\partial}{\partial x}-\i \frac{\partial}{\partial y}\right),\quad  \frac{\partial}{\partial \overline{z}}=\frac{1}{ 2}\left( \frac{\partial}{\partial x}+\i \frac{\partial}{\partial y}\right),
	\end{equation}
	for a complex number $z=x+\i y$ with $x,y \in\mathbb{R}$.
\end{Rem}

For $F\in L_{\mathbb{C}}^2(\Omega, \sigma(Z), P)$, by chaos decomposition \eqref{complex chaos decomposition} and Theorem \ref{Thm 1}, we have
\begin{equation}
	F=\sum_{p=0}^{\infty}\sum_{q=0}^{\infty}\I_{p, q}(f_{p,q})=\sum_{p=0}^{\infty}\sum_{q=0}^{\infty}I_{p+ q}(u_{p,q})+\i\sum_{p=0}^{\infty}\sum_{q=0}^{\infty}I_{p+ q}(v_{p,q}),
\end{equation}
where $I_{p+q}(\cdot)$ is the real $(p+q)$-th Wiener-It\^{o} integral with respect to $W$ and $u_{p,q},v_{p,q}\in(\mathfrak{H}\oplus\mathfrak{H})^{\odot (p+q)}$ are the kernels of the real and imaginary parts of $\I_{p, q}(f_{p,q})$. Thus, $\mathrm{Re}\,F$ and $\mathrm{Im}\,F$ can be uniquely expanded into series of $I_{n}(\cdot)$ as follows
\begin{align}
	\mathrm{Re}\,F&=\sum_{n=0}^{\infty}I_{n}\left( f_n\right),\quad f_n=\sum_{p=0}^{n}u_{p,n-p},\\
	\mathrm{Im}\,F&= \sum_{n=0}^{\infty}I_{n}\left( g_n\right),\quad g_n=\sum_{p=0}^{n}v_{p,n-p}.
\end{align} 	
Combining \eqref{relation r and c derivative} and Stroock's formula \eqref{real Stroock} for real Wiener-It\^o integrals, we obtain the computable expressions of $f_n$ and $g_n$, which can be considered as a generalized Stroock's formula.

\begin{Prop}[Generalized Stroock's formula]\label{generalized Stroock}
	If $F\in \mathscr{D}^{m, 2} \cap \bar{\mathscr{D}}^{m, 2}$ for some $m\geq n$ with the expansions of $\mathrm{Re}F=\sum_{n=0}^{\infty}I_{n}\left( f_n\right)$ and $
	\mathrm{Im}F= \sum_{n=0}^{\infty}I_{n}\left( g_n\right)$, where $f_n,g_n\in \left( \mathfrak{H}\oplus\mathfrak{H}\right)^{\odot n} $ and $I_{n}(\cdot)$ is the real $n$-th Wiener-It\^{o} integral with respect to $W$. Then $f_n$ and $g_n$ are uniquely defined as 
	\begin{equation}
		f_n+\i g_n=\frac{1}{n!}2^{-\frac{n}{2}}\mathbb{E}\left[ \left( \D +\bar{\D},\i\left( \D -\bar{\D}\right) \right) ^{\otimes n} F\right].
	\end{equation}
\end{Prop}

\begin{proof}
	By the Stroock's formula \eqref{real Stroock} and \eqref{relation r and c derivative}, for $n\leq m$, we get
	\begin{align}
		f_n&=\frac{1}{n!}\mathbb{E}\left[ D^n \mathrm{Re}F\right] =\frac{1}{n!}\mathbb{E}\left[ (D^1, D^2)^{\otimes n} \mathrm{Re}F\right]\\
		&=\frac{1}{n!}2^{-\frac{n}{2}}\mathbb{E}\left[ \left( \D +\bar{\D},\i\left( \D -\bar{\D}\right) \right) ^{\otimes n} \mathrm{Re}F \right].
	\end{align}
	By a similar argument, we obtain that
	\begin{equation}
		g_n=\frac{1}{n!}2^{-\frac{n}{2}}\mathbb{E}\left[ \left(\D+\bar{\D},\i\left( \D-\bar{\D}\right)  \right) ^{\otimes n}\mathrm{Im}F\right] .
	\end{equation}
	Then we get the conclusion.
\end{proof}

In Theorem \ref{Thm 1}, we establish a method to get the kernels of the real and imaginary parts of a complex Wiener-It\^o integral by given complete and orthogonal elements in $\mathfrak{H}_{\mathbb{C}}$. However, it is difficult to get explicit expressions for the kernels when the basis of $\mathfrak{H}_{\mathbb{C}}$ cannot be determined. With Proposition \ref{generalized Stroock}, we overcome this difficulty and obtain another expressions of the kernels in terms of complex Malliavin derivative operators $\D$ and $\bar{\D}$.

\begin{Coro}\label{1 generalized Stroock}
	$F=\I_{p,q}(f)$ with $f\in\mathfrak{H}_{\mathbb{C}}^{\odot p}\otimes\mathfrak{H}_{\mathbb{C}}^{\odot q}$ admits the unique representation
	\begin{equation*}
		\I_{p,q}(f) = I_{p+q}(u)+\i I_{p+q}( v),
	\end{equation*}  where $u,v\in(\mathfrak{H}\oplus\mathfrak{H})^{\odot (p+q)}$ are defined as 
	\begin{equation}
		u+\i v =\frac{1}{(p+q)!}2^{-\frac{p+q}{2}}\left( \D +\bar{\D},\i\left( \D -\bar{\D}\right) \right) ^{\otimes (p+q)}  F ,
	\end{equation}
	and $I_{p+q}(\cdot)$ is the real $(p+q)$-th Wiener-It\^{o} integral with respect to $W$.  
\end{Coro}
By Corollary \ref{1 generalized Stroock}, given a complex Wiener-It\^o integral, we can consider it as a two-dimensional random vector whose components are real Wiener-It\^o integrals and expressions for the kernels are explicit. Then 
the asymptotic normality of it can be proved by utilizing the multidimensional version of fourth moment theorem (see \cite[Theorem 5.2.7, Theorem 6.2.3]{nourdin2012normal}). (With \cite[Theorem 3.3]{chen2017fourth}, we can not implement this method since the kernels of real and imaginary parts are unclear.)

\begin{Ex}\label{ex}
	In \cite{ChenHuWang2017}, Chen, Hu and Wang considered the least squares estimator $\hat{\gamma}_T$ of the drift coefficient $\gamma$ for the complex-valued Ornstein-Uhlenbeck processes disturbed by fractional noise, and get the strong consistency and the asymptotic normality of $\hat{\gamma}_T$. The numerator $F_T$ of the statistic $\sqrt{T}(\hat{\gamma}_T-\gamma)$ is a $(1,1)$-th complex Wiener-It\^o integral with respect to a complex fractional Brownian motion with Hurst parameter $H\in\left(\frac{1}{2},\frac{3}{4}\right)$. Namely, assume that  $\gamma \in\mathbb{C}$ is unknown,
	\begin{align}
		F_T&=\I_{1,1}(\psi_{T}(t,s)),\quad \psi_{T}(t,s)=\frac{1}{\sqrt{T}}e^{-\bar{\gamma}(t-s)}\mathbf{1}_{\left\lbrace 0\leq s\leq t\leq T\right\rbrace} ,\\
		\overline{F}_T&=\I_{1,1}(\phi_{T}(t,s)),\quad \phi_{T}(t,s)=\frac{1}{\sqrt{T}}e^{-\gamma(s-t)}\mathbf{1}_{\left\lbrace 0\leq t\leq s\leq T\right\rbrace}.
	\end{align}
	The underlying Hilbert space is defined as 
	\begin{equation}
		\mathfrak{H}:=\left\{f|f: \mathbb{R}_{+} \rightarrow \mathbb{R}, \| f\|_{\mathfrak{H}} ^{2}:=\int_{0}^{\infty} \int_{0}^{\infty} f(s) f(t) \varphi(s, t) \mathrm{d} s \mathrm{d} t<\infty\right\},
	\end{equation}
	with $\alpha_H=H(2H-1)$, $\varphi(s, t)=\alpha_{H}|s-t|^{2 H-2}$ and inner product $$\left\langle  f,g\right\rangle_{\mathfrak{H}}= \int_{0}^{\infty} \int_{0}^{\infty} f(s) g(t) \varphi(s, t) \mathrm{d} s \mathrm{d} t.$$ We complexify $\mathfrak{H}$ in the usual way and denote by $\mathfrak{H}_{\mathbb{C}}$. %For any $f,g\in\mathfrak{H}_{\mathbb{C}}$, 
	%	\begin{equation}
	%		\left\langle  f,g\right\rangle_{\mathfrak{H}_{\mathbb{C}}}= \int_{0}^{\infty} \int_{0}^{\infty} f(s) \overline{g(t)} \varphi(s, t) \mathrm{d} s \mathrm{d} t.
	%	\end{equation}
	In \cite{ChenHuWang2017}, in order to show the asymptotic normality of $F_T$, the authors firstly established some equivalent conditions for the complex fourth moment theorem (see \cite[Theorem 1.3]{ChenHuWang2017}), and then made use of results obtained, calculated accurately 
	\begin{equation}
		\lim\limits_{T\rightarrow\infty}\mathbb{E}\left[ \left|F_T \right| ^2\right],\quad \lim\limits_{T\rightarrow\infty}\mathbb{E}\left[ F_T  ^2\right],
	\end{equation}
	and verified the contraction condition $$\lim\limits_{T\rightarrow\infty}\left\|\psi_{T}\otimes_{0,1}\phi_{T}\right\| _{\mathfrak{H}_{\mathbb{C}}^{\otimes 2}}=0.$$
	Now, by Corollary \ref{1 generalized Stroock}, we get that 
	\begin{equation}
		F_T=F_{1,T}+\i F_{2,T}=I_2(u_T+\i v_T),
	\end{equation}
	where $u_T+\i v_T$ are defined as \eqref{u_T+iv_T} below. This means that we can regard $F_T$ as two-dimensional random vector $\left(F_{1,T}, F_{2,T} \right) $ and utilize the real fourth moment theorem (see \cite[Theorem 5.2.7, Theorem 6.2.3]{nourdin2012normal}) to prove the asymptotic normality of $F_T$. In this way, we need to calculate
	\begin{equation}
		\lim\limits_{T\rightarrow\infty}\mathbb{E}\left[ F_{i,T}F_{j,T}\right],\; i,j=1,2,
	\end{equation}
	and show that
	\begin{equation}
		\lim\limits_{T\rightarrow\infty}\left\| u_T\otimes_{1}u_T\right\| _{\left( \mathfrak{H}\oplus\mathfrak{H}\right)^{\otimes 2} }=0, \quad
		\lim\limits_{T\rightarrow\infty}\left\| v_T\otimes_{1}v_T\right\| _{\left( \mathfrak{H}\oplus\mathfrak{H}\right)^{\otimes 2} }=0.
	\end{equation}
	Next, we derive the expression of the kernel $u_T+\i v_T$ by Corollary \ref{1 generalized Stroock}. Calculating directly, we have
	\begin{align}
		\left(\D+\bar{\D},\i\left( \D-\bar{\D}\right)  \right) ^{\otimes 2}&=\begin{pmatrix}
			\D+\bar{\D}\\\i\left( \D-\bar{\D}\right) 
		\end{pmatrix}
		\otimes
		\begin{pmatrix}
			\D+\bar{\D}\\ \i\left( \D-\bar{\D}\right) 
		\end{pmatrix}\\&=
		\begin{pmatrix}
			\left( \D+\bar{\D} \right) \left( \D+\bar{\D} \right) \\
			\i \left(  \D-\bar{\D}\right) \left(  \D+\bar{\D}\right)\\
			\i \left( \D+\bar{\D} \right) \left(  \D-\bar{\D}\right) \\
			- \left( \D-\bar{\D}\right)\left( \D-\bar{\D} \right)  
		\end{pmatrix}=
		\begin{pmatrix}
			\D_2+\D\bar{\D}+\bar{\D}\D+\bar{\D}^2\\
			\i\left(\D_2+\D\bar{\D}-\bar{\D}\D-\bar{\D}^2\right)\\
			\i\left(\D_2-\D\bar{\D}+\bar{\D}\D-\bar{\D}^2\right)\\
			-\D_2+\D\bar{\D}+\bar{\D}\D-\bar{\D}^2
		\end{pmatrix}.
	\end{align}
	Then according to Corollary \ref{1 generalized Stroock} and the fact that $	\D_2F=\bar{\D}^2F=0$, we get 
	\begin{align}\label{u_T+iv_T}
		&u_T(t,s)+\i v_T(t,s) =\frac{1}{4}\left(\D+\bar{\D},\i\left( \D-\bar{\D}\right)  \right) ^{\otimes 2}F \\
		=&\frac{1}{4}\begin{pmatrix}
			\D\bar{\D}+\bar{\D}\D\\
			\i\left(\D\bar{\D}-\bar{\D}\D\right)\\
			\i\left(-\D\bar{\D}+\bar{\D}\D\right)\\
			\D\bar{\D}+\bar{\D}\D
		\end{pmatrix}F=\frac{1}{4}\begin{pmatrix}
			\psi(t,s)+\psi(s,t)\\
			\i\left(\psi(t,s)-\psi(s,t)\right)\\
			-\i\left(\psi(t,s)-\psi(s,t)\right)\\
			\psi(t,s)+\psi(s,t)
		\end{pmatrix}\\
		=&\frac{1}{4\sqrt{T}}\begin{pmatrix}
			e^{-\gamma |t-s|} \\
			-\i e^{-\gamma |t-s|}\left( \mathbf{1}_{\left\lbrace 0\leq t\leq s\leq T\right\rbrace}-\mathbf{1}_{\left\lbrace 0\leq s\leq t\leq T\right\rbrace}\right) \\
			\i e^{-\gamma |t-s|}\left( \mathbf{1}_{\left\lbrace 0\leq s\leq t\leq T\right\rbrace}-\mathbf{1}_{\left\lbrace 0\leq t\leq s\leq T\right\rbrace}\right)\\
			e^{-\gamma |t-s|}
		\end{pmatrix}.
	\end{align}		
	It is worth noting that the representation theorem (Theorem \ref{Thm 1} or Corollary \ref{1 generalized Stroock}) we established does provide a new method to solve problems concerning complex Wiener-It\^o integrals. The explicit expressions for kernels of real and imaginary parts enable us to do accurate calculations, although the difficulties to overcome and techniques used are similar in the above two methods.		
\end{Ex}	

\section{Kernel representation formula from real to complex Wiener-It\^o integrals}\label{Section4}
\subsection{Uniqueness theorem}\label{Section4.1}

In Section \ref{Section3}, we show how to get the kernels of real and complex parts of a complex multiple Wiener-It\^o integral. Conversely, in this section, we consider a complex random variable, whose real and imaginary parts are two real multiple Wiener-It\^o integrals of the same order, and prove that it can be uniquely expressed as a finite sum of complex Wiener-It\^o integrals.

Let	$\textbf{m}=\left\lbrace m_k \right\rbrace_{k=1}^{\infty},\textbf{n}=\left\lbrace n_k\right\rbrace_{k=1}^{\infty}\in\Lambda$ with $|\textbf{m}|+|\textbf{n}|=p$. 
Let $\textbf{p}=\left\lbrace p_k\right\rbrace _{k=1}^{\infty}:=\textbf{m}+\textbf{n}$, then $|\textbf{p}|=p$. For simplicity of presentation, define $$\tilde{a}_{k,j}:=\frac{\i ^{n_k}}{2^{p_k}}\sum_{r+s=j}\binom{m_k}{r}\binom{n_k}{s}(-1)^s.$$
For $\textbf{j}=\left\lbrace j_k \right\rbrace_{k=1}^{\infty}\in \Lambda$ with $\textbf{j}\leq \textbf{p}$, $\prod_{k=1}^{\infty}\tilde{a}_{k,j_k}$ is a complex number.	Define $g_{l,p-l}\left( \textbf{m},\textbf{n}\right) \in \mathfrak{H}_{\mathbb{C}}^{\odot l}\otimes\mathfrak{H}_{\mathbb{C}}^{\odot( p-l)}$ with $ 0\leq l\leq p  $ as 
\begin{align}
	g_{l,p-l}\left( \textbf{m},\textbf{n}\right)=\sum_{ \textbf{j}\leq \textbf{p},|\textbf{j}|=l}\left(\prod_{i=1}^{\infty}\tilde{a}_{i,j_i} \right) \mathrm{symm}\left( \otimes_{k=1}^{\infty}  e_k^{\otimes j_k}\right) \otimes \mathrm{symm}\left( \otimes_{k=1}^{\infty}   \overline{e}_k^{\otimes (p_k-j_k)} \right) .
\end{align}
Suppose $g_1, g_2\in\left( \mathfrak{H}\oplus\mathfrak{H}\right) ^{\odot p}$, then there exists a unique sequence $$\left\lbrace \tilde{c}(\textbf{m};\textbf{n}):\textbf{m},\textbf{n}\in\Lambda,|\textbf{m}|+|\textbf{n}|=p \right\rbrace \subset \mathbb{C}
$$ satisfying $\sum_{\textbf{m},\textbf{n}\in\Lambda,|\textbf{m}|+|\textbf{n}|=p}\left| \tilde{c}(\textbf{m};\textbf{n})\right| ^2<\infty$ such that 
\begin{equation}\label{expansion for real}
	g_1+\i g_2=\sum_{\textbf{m},\textbf{n}\in\Lambda,|\textbf{m}|+|\textbf{n}|=p} \tilde{c}(\textbf{m};\textbf{n})\mathrm{symm}\left( \otimes_{k=1}^{\infty}\left( u_{1,0}(k)^{\otimes m_k}  \otimes v_{1,0}(k)^{\otimes n_k}\right)\right) .
\end{equation}
\begin{Thm}\label{inverse First proof}
	Suppose $g_1, g_2\in\left( \mathfrak{H}\oplus\mathfrak{H}\right) ^{\odot p}$  and $g_1+\i g_2$ is given by \eqref{expansion for real}. Then $I_{p}\left(g_1\right) +\i I_{p}\left(g_2\right)$ admits the representation
	\begin{align}
		I_{p}\left(g_1\right) +\i I_{p}\left(g_2\right)=
		\sum_{l=0}^{p}\I_{l,p-l}\left(\tilde{ g}_{l,p-l}\right) ,
	\end{align}
	where $\tilde{ g}_{l,p-l}\in\mathfrak{H}_{\mathbb{C}}^{\odot l}\otimes\mathfrak{H}_{\mathbb{C}}^{\odot( p-l)}, 0\leq l\leq p$ are defined as 
	\begin{align}	
		\tilde{ g}_{l,p-l}=\sum_{\textbf{m},\textbf{n}\in\Lambda,|\textbf{m}|+|\textbf{n}|=p}  \tilde{c}(\textbf{m};\textbf{n}) g_{l,p-l}(\textbf{m},\textbf{n}).
	\end{align}
\end{Thm}

We prove Theorem \ref{inverse First proof} in Section \ref{Section5.4}. Moreover, by an induction argument and utilizing recursion formula concerning real multiple Wiener-It\^o integral given by It\^o (see \cite[Equation 3.4]{ito1951real}), we can offer a recursion representation version of Theorem \ref{inverse First proof}. The proof of this result is quite similar to that of Theorem \ref{Thm 1}  and so is omitted.

\subsection{Representation theorem}\label{Section4.2}

For an integer $1\leq j\leq 2^p$, $j-1$ can be uniquely expressed as a binary number
\begin{equation}
	j-1=\sum_{l=1}^{p}a_{jl}2^{l-1},\quad a_{jl}\in\left\lbrace0,1 \right\rbrace .
\end{equation}
Combining the real Stroock's formula \eqref{real Stroock}
with the fact that $\left( D_1,D_2 \right)^{\otimes p}$ is a $2^p$-dimensional column vector defined as
\begin{equation}
	\left( D_1,D_2 \right)^{\otimes p}=\left(D_{a_{j1}+1}D_{a_{j2}+1}\cdots D_{a_{jp}+1} \right) _{1\leq j\leq 2^p},
\end{equation}
we get the following lemma.

\begin{Lemma}\label{inverse kernel}
	Suppose that $I_p(g)$ with $g= \left(g_{j} \right)_{1\leq j\leq 2^p} \in(\mathfrak{H}\oplus\mathfrak{H})^{\odot p}$ is a real $p$-th Wiener-It\^{o} integral with respect to $W$, then
	\begin{equation}
		g_{j}=\frac{1}{p!}D_{a_{j1}+1}D_{a_{j2}+1}\cdots D_{a_{jp}+1}I_{p}(g).
	\end{equation}
\end{Lemma}

For $1\leq j\leq 2^p$ and $0\leq k\leq p$, let $b_{kj}$ denote the number of $1$ in $\left\lbrace a_{j1},\ldots,a_{jk}\right\rbrace $, and $c_{kj}$ denote the number of $1$ in $\left\lbrace a_{j,k+1},\ldots,a_{jp}\right\rbrace $. Set $b_{0j}\equiv0, 1\leq j\leq 2^p$. Define column vectors
\begin{equation}
	V_k=\left(V_{kj} \right)_{1\leq j\leq 2^p} =\left(\left(-\i \right)^{b_{kj}}\i ^{c_{kj}}\right) _{1\leq j\leq 2^p},\quad 0\leq k\leq p . 
\end{equation}

\begin{Thm}\label{inverse expression of kernel}
	For any $g_1= \left(g_{1j} \right)_{1\leq j\leq 2^p} ,g_2=\left(g_{2j} \right)_{1\leq j\leq 2^p}\in(\mathfrak{H}\oplus\mathfrak{H})^{\odot p}$, consider the complex random variable $I_{p}\left(g_1\right) +\i I_{p}\left(g_2\right)$, where $I_{p}(\cdot)$ is the $p$-th real multiple Wiener-It\^{o} integral with respect to $W$. Then 
	\begin{align}
		I_{p}\left(g_1\right) +\i I_{p}\left(g_2\right)=
		\sum_{k=0}^{p}\I_{k,p-k}\left( g_{k,p-k}\right),
	\end{align}
	where $ g_{k,p-k}\in\mathfrak{H}_{\mathbb{C}}^{\odot k}\otimes\mathfrak{H}_{\mathbb{C}}^{\odot( p-k)}, 0\leq k\leq p$, are defined as 
	\begin{equation}\label{eq inverse expression of kernel}
		g_{k,p-k}=\frac{2^{-p/2}p!}{k!(p-k)!}\sum_{j=1}^{2^p}V_{kj}(g_{1j}+\i g_{2j}).
	\end{equation}
\end{Thm}

\begin{proof}
	According to Theorem \ref{inverse First proof}, there uniquely exist $g_{k,p-k}\in \mathfrak{H}_{\mathbb{C}}^{\odot k}\otimes\mathfrak{H}_{\mathbb{C}}^{\odot( p-k)}$ with $k=0,\ldots,p$ such that 
	\begin{equation}
		I_{p}\left(g_1\right) +\i I_{p}\left(g_2\right)=\sum_{k=0}^{p}\I_{k,p-k}(g_{k,p-k}).
	\end{equation}
	Combining complex Stroock's formula \eqref{complex stroock} and Equation \eqref{relation r and c derivative}, we have
	\begin{align}
		g_{k,p-k}&=\frac{1}{k!(p-k)!}\D^k\bar{\D}^{p-k}\left( I_{p}(g_1)+\i I_{p}(g_2)\right)\\
		&=\frac{2^{-p/2}}{k!(p-k)!}(D_1-\i D_2)^k(D_1+\i D_2)^{p-k}\left( I_{p}(g_1)+\i I_{p}(g_2)\right)\\ &=\frac{2^{-p/2}p!}{k!(p-k)!}\sum_{j=1}^{2^p}V_{kj}(g_{1j}+\i g_{2j}),
	\end{align}
	where the last equality follows from Lemma \ref{inverse kernel} and the fact that 
	$$(D_1-\i D_2)^k(D_1+\i D_2)^{p-k}=\sum_{j=1}^{2^p}V_{kj}D_{a_{j1}+1}D_{a_{j2}+1}\cdots D_{a_{jp}+1} .$$
\end{proof}

By Theorem \ref{inverse expression of kernel}, we get the following theorem. This theorem shows that if the kernels $g_1= \left(g_{1j} \right)_{1\leq j\leq 2^p} ,g_2=\left(g_{2j} \right)_{1\leq j\leq 2^p}\in(\mathfrak{H}\oplus\mathfrak{H})^{\odot p}$ satisfy condition \eqref{condition 1}, then the two-dimensional random vector $\left( I_{p}(g_1), I_{p}(g_2)\right) $ can be regarded as a complex multiple Wiener-It\^o integral. This means that we can utilize the theory of complex multiple Wiener-It\^o integrals to solve the problems concerning two-dimensional random vectors whose components are real multiple Wiener-It\^o integrals of the same order.

\begin{Thm}\label{2 to 1}
	Given a two-dimensional random vector $\left( I_{p}(g_1), I_{p}(g_2)\right) $ with $g_1= \left(g_{1j} \right)_{1\leq j\leq 2^p} ,g_2=\left(g_{2j} \right)_{1\leq j\leq 2^p}\in(\mathfrak{H}\oplus\mathfrak{H})^{\odot p}$, whose components are real multiple Wiener-It\^o integrals with respect to $W$. If there exists a unique $0\leq k\leq p$ such that
	\begin{equation}\label{condition 1}
		\begin{cases}
			\sum_{j=1}^{2^p}V_{kj}(g_{1j}+\i g_{2j})\neq0,\\
			\sum_{j=1}^{2^p}V_{lj}(g_{1j}+\i g_{2j})=0, \;l\neq k, 0\leq l\leq p,
		\end{cases}
	\end{equation}
	then $$I_{p}\left(g_1\right) +\i I_{p}\left(g_2\right)=\I_{k,p-k}(g_{k,p-k}),$$ where $g_{k,p-k}$ is defined as \eqref{eq inverse expression of kernel}.    
\end{Thm}

\begin{Ex}\label{inverse ex}
	Back to Example \ref{ex}. $u_T(t,s)+\i v_T(t,s)\in (\mathfrak{H}\oplus\mathfrak{H})^{\odot 2 }$ is given by \eqref{u_T+iv_T}.
	For $p=2$,
	\begin{equation}
		V_0=\left( 1,\i,\i,-1\right) ^T,\,V_1=\left( 1,-\i,\i,1\right) ^T,\,V_2=\left( 1,-\i,-\i,-1\right) ^T.
	\end{equation}
	According to Theorem \ref{inverse expression of kernel} , $I_2(u_T)+\i I_2( v_T)$ can be uniquely expressed as 
	\begin{equation}
		I_2(u_T)+\i I_2( v_T)=\I_{0,2}(g_{0,2})+\I_{1,1}(g_{1,1})+\I_{2,0}(g_{2,0}),
	\end{equation}
	where
	\begin{align}
		g_{0,2}&=\frac{e^{-\gamma |t-s|}}{8\sqrt T}\left( 1-\mathbf{1}_{\left\lbrace 0\leq s\leq t\leq T\right\rbrace}+\mathbf{1}_{\left\lbrace 0\leq t\leq s\leq T\right\rbrace}- \mathbf{1}_{\left\lbrace 0\leq t\leq s\leq T\right\rbrace}+\mathbf{1}_{\left\lbrace 0\leq s\leq t\leq T\right\rbrace} -1\right)=0,\\
		g_{1,1}&=\frac{e^{-\gamma |t-s|}}{2\sqrt T}\left( 1+
		\mathbf{1}_{\left\lbrace 0\leq s\leq t\leq T\right\rbrace}-\mathbf{1}_{\left\lbrace 0\leq t\leq s\leq T\right\rbrace}\right) = 	\frac{1}{\sqrt{T}} e^{-\bar{\gamma}(t-s)}\mathbf{1}_{\left\lbrace 0\leq s\leq t\leq T\right\rbrace}=\psi_{T}(t,s),\\
		g_{2,0}&=\frac{e^{-\gamma |t-s|}}{8\sqrt T}\left( 1+\mathbf{1}_{\left\lbrace 0\leq s\leq t\leq T\right\rbrace}-\mathbf{1}_{\left\lbrace 0\leq t\leq s\leq T\right\rbrace}+
		\mathbf{1}_{\left\lbrace 0\leq t\leq s\leq T\right\rbrace}-\mathbf{1}_{\left\lbrace 0\leq s\leq t\leq T\right\rbrace} -1\right) =0.
	\end{align}
	That is to say, the condition \eqref{condition 1} is satisfied. Therefore,
	\begin{equation}
		I_2(u_T)+\i I_2( v_T)=\I_{1,1}(g_{1,1})=\I_{1,1}(\psi_{T}).
	\end{equation}
\end{Ex}

\begin{Rem}
	Combining the main results we established in Section \ref{Section3} and Section \ref{Section4}, we derive that 
	\begin{align}\label{chao decom}
		L_{\mathbb{C}}^2(\Omega, \sigma(Z), P)&=\bigoplus_{p=0}^{\infty}\bigoplus_{q=0}^{\infty}\mathscr{H}_{p,q}(Z)\\&=\bigoplus_{p=0}^{\infty}\bigoplus_{k+l=p}\mathscr{H}_{k,l}(Z)=\bigoplus_{p=0}^{\infty}\left( \mathcal{H}_p(W)+\i \mathcal{H}_p(W)\right).
	\end{align}
	(\eqref{chao decom} was proved by an existence proof in \cite{chen2017fourth}. We further clearly characterize the kernels of real and complex Wiener-It\^o integrals in this paper.)
	We can understand chaos decomposition \eqref{chao decom} from the perspective on the characteristic subspace of real and complex Ornstein-Uhlenbeck operators. Specifically, $\mathcal{H}_{p,q}(Z)$ is the characteristic subspace associated with the eigenvalue $-(p+q)\cos\theta-\i (p-q)\sin\theta$ of the infinitesimal generator $\L_{\theta}=e^{\mathrm{i} \theta} \L+e^{-\mathrm{i} \theta} \bar{\L} $ of the complex Ornstein-Uhlenbeck semigroup $\left\{\mathcal{T}_{t}\right\}$, where $\L$ and $\bar{\L}$ are complex Ornstein-Uhlenbeck operators and $\theta\in\left(-\frac{\pi}{2},\frac{\pi}{2} \right) $ is fixed, see \cite{chenliu2019} for more details. Given $I_p(g_1+\i g_2) \in \mathcal{H}_p(W)+\i \mathcal{H}_p(W)$ with $g_1,g_2\in(\mathfrak{H}\oplus\mathfrak{H})^{\odot p}$, Theorem \ref{2 to 1} shows that if the kernels $g_1$ and $g_2$ satisfy condition \eqref{condition 1}, then $I_p(g_1+\i g_2)$ is a eigenfunction associated with the eigenvalue $-p\cos\theta-\i (2k-p)\sin\theta$ of $\L_{\theta} $. Moreover, $\mathcal{H}_p(W)$ is the characteristic subspace associated with the eigenvalue $-p$ of the infinitesimal generator $L$ of the real Ornstein-Uhlenbeck semigroup $\left\{T_{t}\right\}$, see \cite[Section 1.4]{nualart2006malliavin} for more details. Given $\I_{p,q}(f)$ with $f\in\mathfrak{H}_{\mathbb{C}}^{\odot p}\otimes\mathfrak{H}_{\mathbb{C}}^{\odot q}$, by Theorem \ref{First proof}, Theorem \ref{Thm 1} or Corollary \ref{1 generalized Stroock}, $\mathrm{Re}\,\I_{p,q}(f)$ and $\mathrm{Im}\,\I_{p,q}(f)$ are eigenfunctions associated with the eigenvalue $-(p+q)$ of $L$.
	
\end{Rem}

\section{Proofs of main results}\label{Section5}
\subsection{Proof of Lemma \ref{lemma_basis}}\label{Section5.1}
\begin{proof}[Proof of Lemma \ref{lemma_basis}]
	Since $\left\{ e_k\right\}_{k\geq1} $ are orthogonal in $\mathfrak{H}_{\mathbb{C}}$ and $\left\|e_k \right\|^2_{\mathfrak{H}_{\mathbb{C}}}=2$, we have 
	\begin{align}
		2\delta_{kj}&=\left\langle e_k, e_j \right\rangle _{\mathfrak{H}_{\mathbb{C}}}
		=\left\langle e_k^1+\mathrm{i}e_k^2, e_j^1+\mathrm{i}e_j^2 \right\rangle _{\mathfrak{H}_{\mathbb{C}}}\\
		&=\left(\left\langle e_k^1, e_j^1 \right\rangle _{\mathfrak{H}}+\left\langle e_k^2, e_j^2 \right\rangle _{\mathfrak{H}}\right)+\mathrm{i}\left(\left\langle e_k^2, e_j^1 \right\rangle _{\mathfrak{H}}-\left\langle e_k^1, e_j^2 \right\rangle _{\mathfrak{H}}\right),
	\end{align}
	where $\delta_{kj} = \mathbf{1}_{\{k=j\}}$.
	It implies that for any $k$ and $j$,
	\begin{equation*}
		\begin{aligned}
			\left\langle e_k^1, e_j^1 \right\rangle _{\mathfrak{H}}+\left\langle e_k^2, e_j^2 \right\rangle _{\mathfrak{H}}&=2\delta_{kj},\\
			\left\langle e_k^2, e_j^1 \right\rangle _{\mathfrak{H}}-\left\langle e_k^1, e_j^2 \right\rangle _{\mathfrak{H}}&=0.
		\end{aligned}
	\end{equation*}
	Then, 
	\begin{equation*}
		\begin{aligned}
			\left\langle u_{1,0}(k),u_{1,0}(j)\right\rangle _{\mathfrak{H}\oplus\mathfrak{H}}&=\frac{1}{2}\left(\left\langle e_k^1, e_j^1 \right\rangle _{\mathfrak{H}}+\left\langle e_k^2, e_j^2 \right\rangle _{\mathfrak{H}}\right)=\delta_{kj},\\
			\left\langle v_{1,0}(k),v_{1,0}(j)\right\rangle _{\mathfrak{H}\oplus\mathfrak{H}}&=\frac{1}{2}\left(\left\langle e_k^2, e_j^2 \right\rangle _{\mathfrak{H}}+\left\langle e_k^1, e_j^1 \right\rangle _{\mathfrak{H}}\right)=\delta_{kj},\\
			\left\langle u_{1,0}(k),v_{1,0}(j)\right\rangle _{\mathfrak{H}\oplus\mathfrak{H}}&=\frac{1}{2}\left(\left\langle e_k^1, e_j^2 \right\rangle _{\mathfrak{H}}-\left\langle e_k^2, e_j^1 \right\rangle _{\mathfrak{H}}\right)=0,
		\end{aligned}
	\end{equation*}
	which shows that $\left\{ u_{1,0}(k),v_{1,0}(k)\right\}_{k\geq1} $ are orthonormal in $\mathfrak{H}\oplus\mathfrak{H}$.
	
	Next, we show the completeness of $\left\{ u_{1,0}(k),v_{1,0}(k)\right\}_{k\geq1} $. Suppose that $w=(w_1,w_2)\in\mathfrak{H}\oplus\mathfrak{H}$ satisfies $\left\langle w,u_{1,0}(k)\right\rangle _{\mathfrak{H}\oplus\mathfrak{H}}=0$ and $\left\langle w,v_{1,0}(k)\right\rangle _{\mathfrak{H}\oplus\mathfrak{H}}=0$ for any $k\geq1$. That is,
	\begin{equation*}
		\begin{aligned}
			\left\langle w,u_{1,0}(k)\right\rangle _{\mathfrak{H}\oplus\mathfrak{H}}=\frac{1}{\sqrt2}\left(\left\langle w_1,e_k^1 \right\rangle _{\mathfrak{H}}-\left\langle w_2, e_k^2 \right\rangle _{\mathfrak{H}}\right)=0,\\
			\left\langle w,v_{1,0}(k)\right\rangle _{\mathfrak{H}\oplus\mathfrak{H}}=\frac{1}{\sqrt2}\left(\left\langle w_1,e_k^2 \right\rangle _{\mathfrak{H}}+\left\langle w_2, e_k^1 \right\rangle _{\mathfrak{H}}\right)=0.
		\end{aligned}
	\end{equation*} 
	Let $\tilde{w}=w_2+\mathrm{i}w_1$, then 
	\begin{align}
		\left\langle \tilde{w}, e_k \right\rangle _{\mathfrak{H}_{\mathbb{C}}}
		&=\left\langle w_2+\mathrm{i}w_1, e_k^1+\mathrm{i}e_k^2 \right\rangle _{\mathfrak{H}_{\mathbb{C}}}\\
		&=\left(\left\langle w_2, e_k^1 \right\rangle _{\mathfrak{H}}+\left\langle w_1, e_k^2 \right\rangle _{\mathfrak{H}}\right)+\mathrm{i}\left(\left\langle w_1, e_k^1 \right\rangle _{\mathfrak{H}}-\left\langle w_2, e_k^2 \right\rangle _{\mathfrak{H}}\right)\\
		&=0,
	\end{align}
	which indicates that $\tilde{w}=0$ by the completeness of $\left\{ e_k\right\}_{k\geq1} $ in $\mathfrak{H}_{\mathbb{C}}$. Therefore, $w=(w_1,w_2)=0$ in $\mathfrak{H}\oplus\mathfrak{H}$.
\end{proof}

\subsection{Proofs of Lemma \ref{lemma repre for sub basis}, Proposition \ref{Prop repre for basis} and Theorem \ref{First proof}} \label{Section5.2}

From now on, for the real multiple Wiener-It\^{o} integral $I_{p}(\cdot)$ with respect to $W$, if the kernel is complex with the form $f+\i g$, where $f,g\in (\mathfrak{H}\oplus\mathfrak{H})^{\odot p}$, we set $$I_{p}(f+\i g)=I_{p}(f)+\i I_{p}(g).$$

\begin{proof}[Proof of Lemma \ref{lemma repre for sub basis}]

	By the relation between real and complex Hermite polynomials \eqref{real and complex Hermite}, we have
	\begin{align}
		&\I_{p, q}\left(e_1^{\otimes p}\otimes\overline{e}_1^{\otimes q}\right)=J_{p, q}\left(Z(e_1)\right)\\
		=&\sum_{j=0}^{p+q} \mathrm{i}^{p+q-j} \sum_{r+s=j}\binom{p}{r}\binom{q}{s}(-1)^{q-s}H_j\left( \mathrm{Re} Z(e_1)\right) H_{p+q-j}\left( \mathrm{Im} Z(e_1)\right)\\
		=&\sum_{j=0}^{p+q} \mathrm{i}^{p+q-j} \sum_{r+s=j}\binom{p}{r}\binom{q}{s}(-1)^{q-s}H_j\left( W\left(u_{1,0}(1) \right) \right) H_{p+q-j}\left( W\left(v_{1,0}(1) \right)\right)\\
		=&\sum_{j=0}^{p+q} \mathrm{i}^{p+q-j} \sum_{r+s=j}\binom{p}{r}\binom{q}{s}(-1)^{q-s}I_j\left(u_{1,0}(1)^{\otimes j} \right) I_{p+q-j}\left( v_{1,0}(1) ^{\otimes (p+q-j)}\right).
	\end{align}
	According to the definition of real Wiener-It\^o integral \eqref{def of real integral},
	\begin{equation}
		I_j\left(u_{1,0}(1)^{\otimes j} \right) I_{p+q-j}\left( v_{1,0}(1) ^{\otimes (p+q-j)}\right)
		=I_{p+q}\left(u_{1,0}(1)^{\otimes j}\otimes v_{1,0}(1) ^{\otimes (p+q-j)}\right).
	\end{equation}
	Then the proof is completed.
\end{proof}

\begin{proof}[Proof of Proposition \ref{Prop repre for basis}]
	Using the definitions of real and complex Wiener-It\^o integrals, \eqref{def of real integral} and \eqref{def of complex integral}, respectively, we have
	\begin{align}
		&\I_{p,q}\left(\mathrm{symm}\left( \otimes_{k=1}^{\infty}e_k^{\otimes p_k}\right) \otimes  \mathrm{symm}\left( \otimes_{k=1}^{\infty}\overline{e}_k^{\otimes q_k}\right)\right)\\=& \prod_{k=1}^{\infty}\I_{p_k,q_k}\left( e_k^{\otimes p_k} \otimes   \overline{e}_k^{\otimes q_k}\right) \\
		\overset{\eqref{repre for sub basis}}{=}&\prod_{k=1}^{\infty}\left(  \sum_{j=0}^{p_k+q_k} a_{k,j}I_{p_k+q_k}\left( u_{1,0}(k)^{\otimes j}\otimes v_{1,0}(k)^{\otimes (p_k+q_k-j)}\right)\right) \\
		=&\sum_{\textbf{j}\leq \textbf{p}+\textbf{q}}\left( \prod_{l=1}^{\infty}a_{l,j_l}\right) I_{p+q}\left(\otimes_{k=1}^{\infty} \left( u_{1,0}(k)^{\otimes j_k}\otimes v_{1,0}(k)^{\otimes (p_k+q_k-j_k)} \right) \right) \\
		=&I_{p+q}\left(u(\textbf{p},\textbf{q})+\i v(\textbf{p},\textbf{q})\right).
	\end{align}
\end{proof}

\begin{proof}[Proof of Theorem \ref{First proof}]
	For $f\in\mathfrak{H}_{\mathbb{C}}^{\odot p}\otimes\mathfrak{H}_{\mathbb{C}}^{\odot q}$ with the expansion given by \eqref{another expansion},
	\begin{align}
		\I_{p,q}\left(f\right)
		&=\sum_{\textbf{p},\textbf{q}\in\Lambda,|\textbf{p}|=p, |\textbf{q}|=q}c(\textbf{p};\textbf{q})\I_{p,q}\left(\mathrm{symm}\left( \otimes_{k=1}^{\infty}e_k^{\otimes p_k}\right) \otimes  \mathrm{symm}\left( \otimes_{k=1}^{\infty}\overline{e}_k^{\otimes q_k}\right)\right)\\
		&=\sum_{\textbf{p},\textbf{q}\in\Lambda,|\textbf{p}|=p, |\textbf{q}|=q}c(\textbf{p};\textbf{q})  I_{p+q}\left( u(\textbf{p},\textbf{q})+\i v(\textbf{p},\textbf{q})\right) \\
		&=\sum_{\textbf{p},\textbf{q}\in\Lambda,|\textbf{p}|=p, |\textbf{q}|=q}I_{p+q}\left(  c(\textbf{p};\textbf{q})\left(u(\textbf{p},\textbf{q})+\i v(\textbf{p},\textbf{q}) \right) \right)  .
	\end{align}
	We complete the proof.
\end{proof}

\subsection{Proofs of Theorem \ref{prop_represent} and Theorem \ref{Thm 1} }\label{Section5.3}

For $1\leq r\leq q+1$, we write
\begin{equation}
	\I_{p,q}(\k,\widehat{\j}_r) :=  \I_{p,q}(e_{k_1}\otimes\cdots\otimes e_{k_p}\otimes \overline{e_{j_1}}\otimes\cdots \overline{e_{j_{r-1}}}\otimes\overline{e_{j_{r+1}}}\otimes \cdots\otimes\overline{e_{j_{q+1}}}).
\end{equation}
For $1\leq l\leq p+1$, we also define $\I_{p,q}(\widehat{\k}_l,\j)$ in the same way. The following useful lemma is rephrased from \cite[Theorem 9]{ito1952complex}. 
\begin{Lemma}
	For $p\geq0$, $q\geq0$, we have
	\begin{align}
		\I_{p+1,q}\left(\k,\j\right)
		&= \I_{p,q}\left(\k,\j\right) \I_{1,0}\left(e_{k_{p+1}}\right)
		- \mathbf{1}_{\left\{ q>0 \right\}} \sum_{r=1}^{q}2\delta_{j_r,k_{p+1}} \I_{p,q-1}\left(\k, \widehat{\j}_r\right),\label{I_recurrence_p} \\
		\I_{p,q+1}\left(\k,\j\right)
		&=\I_{p,q}\left(\k,\j\right)\I_{0,1}\left(\overline{e_{j_{q+1}}}\right)-\mathbf{1}_{\left\{ p>0\right\} }\sum_{l=1}^{p}2\delta_{k_l,j_{q+1}}\I_{p-1,q}\left(\widehat{\k}_l,\j\right) 	\label{I_recurrence_q},
	\end{align}
	where $\mathbf{1}_{A}$ is the indicator function of a set $A$ and $\delta_{k,j}=\mathbf{1}_{\left\lbrace k=j\right\rbrace }$.
\end{Lemma} 

For simplicity, we write
\begin{equation}\label{kernel-1}
	\begin{aligned}
		u_{p,q}\left(\k,\widehat{\j}_r\right) &:= u_{p,q}\left(k_1,\dots, k_p;j_1,\dots, j_{r-1}, j_{r+1}, \cdots, j_{q+1}\right), \quad \mbox{for } 1\leq r\leq q+1, \\
		u_{p,q}\left(\widehat{\k}_l,\j\right) &:= u_{p,q}\left(k_1,\dots, k_{l-1}, k_{l+1}, \dots, k_{p+1};j_1,\cdots, j_{q}\right),  \quad \mbox{for } 1\leq l\leq p+1,\\
		v_{p,q}\left(\k,\widehat{\j}_r\right) &:= v_{p,q}\left(k_1,\dots, k_p;j_1,\dots, j_{r-1}, j_{r+1}, \cdots, j_{q+1}\right), \quad \mbox{for } 1\leq r\leq q+1, \\
		v_{p,q}\left(\widehat{\k}_l,\j\right) &:= v_{p,q}\left(k_1,\dots, k_{l-1}, k_{l+1}, \dots, k_{p+1};j_1,\cdots, j_{q}\right),  \quad \mbox{for } 1\leq l\leq p+1.
	\end{aligned}
\end{equation}
By \eqref{u_recurrence} and \eqref{v_recurrence}, we get that
\begin{align}\label{(u+iv)_recurrence}
	w_{p,q}(\k,\j)&:=u_{p,q}(\k,\j)+\i v_{p,q}(\k,\j)\\&=\left( u_{p-1,q}(\k,\j)+\i v_{p-1,q}(\k,\j)\right) \tilde{\otimes}\left( u_{1,0}(k_p)+\i v_{1,0}(k_p)\right) \\
	&=\left( u_{p,q-1}(\k,\j)+\i v_{p,q-1}(\k,\j)\right) \tilde{\otimes}\left( u_{0,1}(j_q)+\i v_{0,1}(j_q)\right) .
\end{align}	
Define $w_{1,0}(j):=u_{1,0}(j)+\i v_{1,0}(j)$ and $w_{0,1}(j):=u_{0,1}(j)+\i v_{0,1}(j)$.
\begin{proof}[Proof of Theorem \ref{prop_represent}]
	The proof is by induction.
	We use $\A(p,q)$ to denote the equation
	\begin{equation}
		\I_{p,q}\left(\k,\j\right)= I_{p+q}\left(u_{p,q}\left(\k,\j\right)+\mathrm{i} v_{p,q}\left(\k,\j\right)\right)=I_{p+q}\left(w_{p,q}\left(\k,\j\right)\right). 
	\end{equation} 
	For $p\geq0$, $q\geq0$, $p+q>0$, let $\B(p,q)$ denote the equation
	\begin{equation}
		(p+q)I_{p+q-1}\left( w_{p,q}(\k,\j) \tilde{\otimes}_1 w_{1,0}(k_{p+1}) \right)  =\mathbf{1}_{\left\{q>0\right\}} \sum_{r=1}^{q}2\delta_{j_r,k_{p+1}}\I_{p,q-1}(\k, \widehat{\j}_r).
	\end{equation}
	
	We will prove $\A(p,q)$ and $\B(p,q)$ hold by induction with respect to $n:=p+q$ in three steps. See \autoref{Picture0} for complete ideas of proof and \autoref{Picture} for details.
	\begin{figure}[htbp]
		\centering
		\includegraphics[width=10cm]{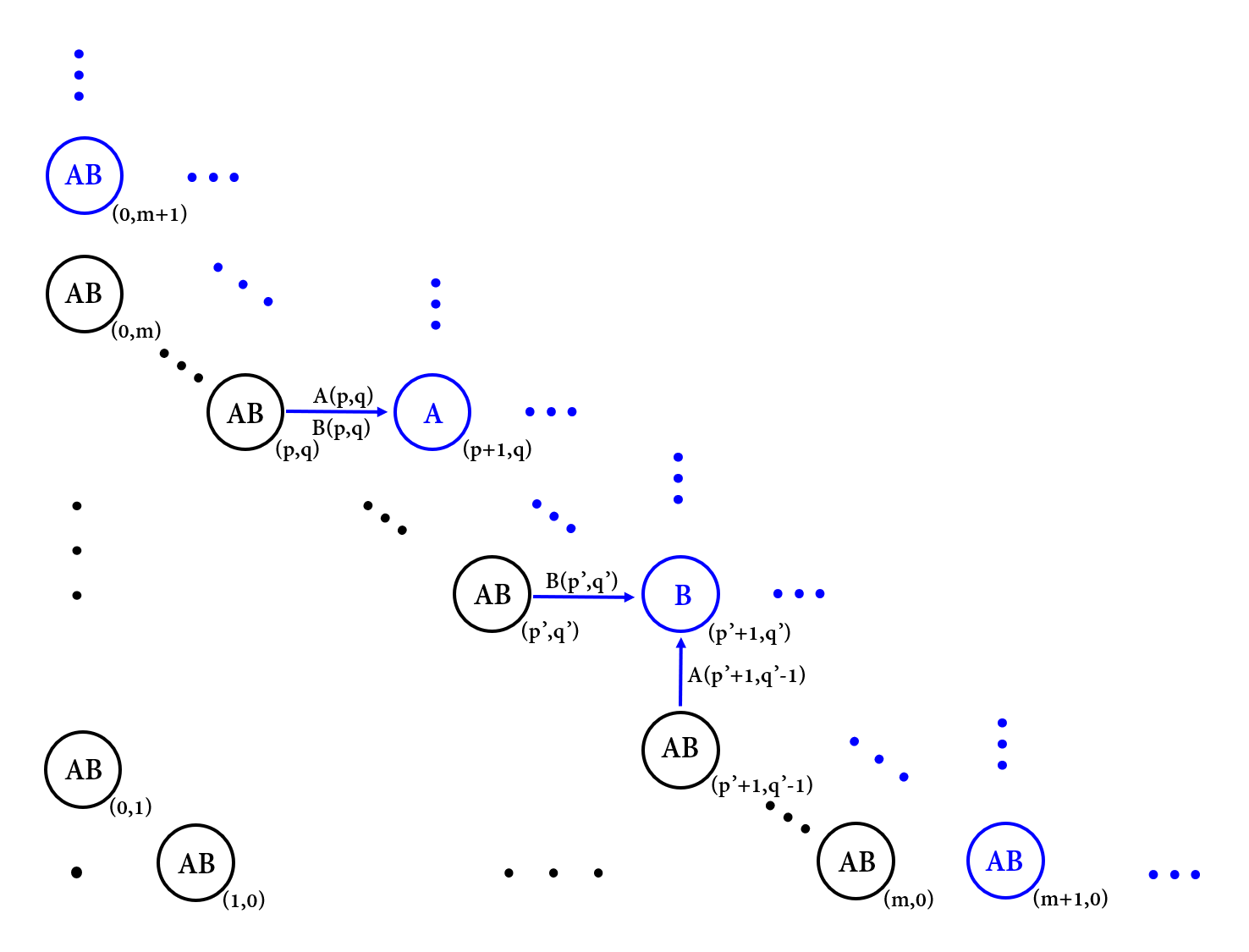}
		\caption{{\small The black part represents that conditions $A$ and $B$ hold for $(p,q)$ with $p+q\leq m$. The blue part represents conditions $A$ and $B$ for $(p,q)$ with $p+q\geq m+1$. We prove the blue part by induction with respect to $p+q$.}}
		\label{Picture0}
	\end{figure}

	\begin{figure}[htbp]
		\centering
		\subfigure[Step 1]{\includegraphics[width=3cm]{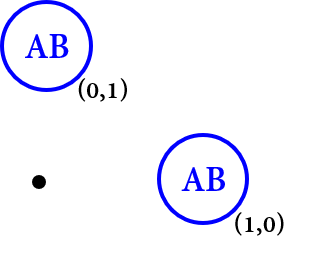}}
		\hspace{.3in}
		\subfigure[Step 2]{\includegraphics[width=3.5cm]{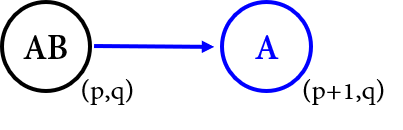}}
		\hspace{.3in}
		\subfigure[Step 3]{\includegraphics[width=3.5cm]{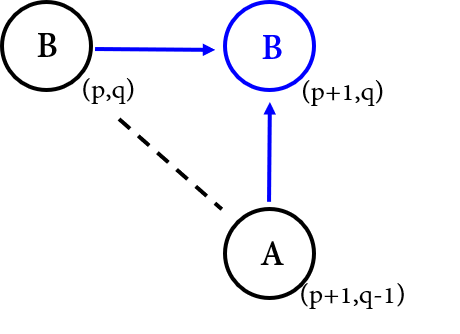}}
		\caption{\small{In Step 1, we prove $A(p,q)$ and $B(p,q)$ for $p+q=1$. In Step 2, we prove $A(p+1,q)$ with $A(p,q)$ and $B(p,q)$ hold. In Step 3, we prove $B(p+1,q)$ with $A(p+1,q-1)$ and $B(p,q)$ hold.}}
		\label{Picture}
	\end{figure} 	
	
	\textbf{Step 1}
	Suppose $n=p+q=1$. For $(p,q)=(1,0)$, \eqref{I_10} implies $\A(1,0)$.
	By \eqref{uv_10}, $$\I_{0,1}(\overline{e_j})=\overline{\I_{1,0}({e_j})}=I_1(u_{1,0}(j))-\mathrm{i}I_1(v_{1,0}(j))=I_1(u_{0,1}(j)+\mathrm{i}v_{0,1}(j)).$$ 
	It follows that $\A(0,1)$ holds.
	
	A direct calculation shows that
	\begin{align}
		&I_{0}\left( w_{0,1}(j_1)\tilde{\otimes}_1w_{1,0}(k_{1})  \right) \\
		=&\left( u_{1,0}(j_1)-\i v_{1,0}(j_1)\right)\tilde{\otimes}_1\left(u_{1,0}(k_{1}+\i v_{1,0}(k_{1} \right)\\
		=&u_{1,0}(j_1)\tilde{\otimes}_1u_{1,0}(k_{1}) + v_{1,0}(j_1)\tilde{\otimes}_1v_{1,0}(k_{1}) + \mathrm{i}\left(u_{1,0}(j_1)\tilde{\otimes}_1v_{1,0}(k_{1})-v_{1,0}(j_1)\tilde{\otimes}_1u_{1,0}(k_{1})\right)\\
		=&2\delta_{j_1,k_1}.
	\end{align}
	Therefore, we have $\B(0,1)$. Similarly, we get that
	\begin{equation}		
		I_{0}\left(w_{1,0}(k_1)\tilde{\otimes}_1 w_{1,0}(k_{2}) \right) =0,	
	\end{equation}		
	which implies $\B(1,0)$.
	
	\textbf{Step 2}
	Suppose $\A(p,q)$ and $\B(p,q)$ hold for $n=p+q\leq m$ with $m\geq 1$.
	In this step, we will prove that $\A(p,q)$ holds for $n = m+1$. 
	
	It suffices to prove $\A(p+1,q)$ and $\A(p,q+1)$ hold for $p+q=m$. By \eqref{I_recurrence_p}, $\A(p,q)$, $\A(1,0)$ and \eqref{Product_formula}, we have
	\begin{align}
		&\I_{p+1,q}(\k,\j)\\
		=&\I_{p,q}(\k,\j)\I_{1,0}(e_{k_{p+1}})
		-\mathbf{1}_{\left\{ q>0 \right\}} \sum_{r=1}^{q}2\delta_{j_r,k_{p+1}} \I_{p,q-1}(\k,\widehat{\j}_r)\\
		=&I_{p+q}(w_{p,q}(\k,\j)) I_1(w_{1,0}(k_{p+1}))-\mathbf{1}_{\left\{ q>0\right\} }\sum_{r=1}^{q}2\delta_{j_r,k_{p+1}}\I_{p,q-1}(\k,\widehat{\j}_r)\\
		=&\sum_{r=0}^{1}r!\binom{p+q}{r}\binom{1}{r} I_{p+q+1-2r}\left(
		w_{p,q}(\k,\j) \tilde{\otimes}_rw_{1,0}(k_{p+1}) \right) \\&-\mathbf{1}_{\left\{ q>0\right\} }\sum_{r=1}^{q}2\delta_{j_r,k_{p+1}}\I_{p,q-1}(\k,\widehat{\j}_r)\\
		=&I_{p+q+1}\left(
		w_{p,q}(\k,\j) \tilde{\otimes}w_{1,0}(k_{p+1}) \right) +\left( p+q\right)  I_{p+q+1-2}\left(
		w_{p,q}(\k,\j) \tilde{\otimes}_1w_{1,0}(k_{p+1}) \right)\\
		&-\mathbf{1}_{\left\{ q>0\right\} }\sum_{r=1}^{q}2\delta_{j_r,k_{p+1}}\I_{p,q-1}(\k,\widehat{\j}_r)\\
		=&I_{p+q+1}\left(w_{p+1,q}(\k,\j)\right),
	\end{align}
	where the last equality follows from $\B(p,q)$ and \eqref{(u+iv)_recurrence}. Therefore, $\A(p+1,q)$ holds.

	By \eqref{uv_conjugate} and $\B(q,p)$, we have
	\begin{align}
		&(p+q)I_{p+q-1}\left(w_{p,q}(\k,\j) \tilde{\otimes}_1w{0,1}(j_{q+1})\right)  
		=(p+q)I_{p+q-1}\left( \overline{w_{q,p}(\j,\k)} \tilde{\otimes}_1 \overline{w_{1,0}(j_{q+1})} \right) 
		\\=&\mathbf{1}_{\left\{ p>0\right\} }\sum_{l=1}^{p}2\delta_{k_l,j_{q+1}}\overline{\I_{q,p-1}(\j,\widehat{\k}_l)}
		=\mathbf{1}_{\left\{ p>0\right\} }\sum_{l=1}^{p}2\delta_{k_l,j_{q+1}}\I_{p-1,q}(\widehat{\k}_l,\j).
	\end{align}
	Combining this equation and \eqref{I_recurrence_q}, we obtain $\A(p,q+1)$ by a similar argument of $\A(p+1,q)$. 
	
	\textbf{Step 3} 
	Suppose $\A(p,q)$ holds for $n=p+q\leq m+1$ and $\B(p,q)$ holds for $p+q\leq m$. In this step, we will prove $\B(p,q)$ holds for $n=m+1$.
	
	It suffices to show that $\B(p+1,q)$ and $\B(p,q+1)$ hold for $p+q=m$, where the equation $\B(p+1,q)$ is 
	\begin{align}
		&(p+q+1)I_{p+q}\left(w_{p+1,q}(\k,\j)  \tilde{\otimes}_1  w_{1,0}(k_{p+2}) \right)   
		=\mathbf{1}_{\left\{ q>0\right\}} \sum_{r=1}^{q}2\delta_{j_r,k_{p+2}}\I_{p+1,q-1}(\k,\widehat{\j}_r).
	\end{align}
	By $\A(p+1,q-1)$, we have 
	\begin{equation}
		\I_{p+1,q-1}(\k,\widehat{\j}_r) = I_{p+q}\left( w_{p+1,q-1}(\k,\widehat{\j}_r) \right).
	\end{equation}
	To prove $\B(p+1,q)$, it suffices to show that
	\begin{align}\label{kernel_contraction}
		(p+q+1) \left(w_{p+1,q}(\k,\j) \tilde{\otimes}_1  w_{1,0}(k_{p+2})\right)
		=\mathbf{1}_{\left\{ q>0\right\} }\sum_{r=1}^{q}2\delta_{j_r,k_{p+2}} w_{p+1,q-1}(\k,\widehat{\j}_r) .
	\end{align}
	By \eqref{(u+iv)_recurrence}, we have
	\begin{align}
		w_{p+1,q}(\k,\j) \tilde{\otimes}_1  w_{1,0}(k_{p+2})
		=&\left[w_{p,q}(\k,\j) \tilde{\otimes}w_{1,0}(k_{p+1})\right] \tilde{\otimes}_1w_{1,0}(k_{p+2})\\
		=&\frac{p+q}{p+q+1}\left[w_{p,q}(\k,\j) \tilde{\otimes}_1w_{1,0}(k_{p+2}) \right] \tilde{\otimes}w_{1,0}(k_{p+1})\\
		=&\frac{1}{p+q+1}\mathbf{1}_{\left\{ q>0\right\} }\sum_{r=1}^{q}2\delta_{j_r,k_{p+2}}
		w_{p,q-1}(\k,\widehat{\j}_r)\tilde{\otimes}w_{1,0}(k_{p+1}) \\
		=&\frac{1}{p+q+1}\mathbf{1}_{\left\{ q>0\right\} }\sum_{r=1}^{q}2\delta_{j_r,k_{p+2}}w_{p+1,q-1}(\k,\widehat{\j}_r),
	\end{align}
	where the second to last equality follows from $\B(p,q)$ and the last equality follows from \eqref{(u+iv)_recurrence}. This means \eqref{kernel_contraction} and thus $\B(p+1,q)$ holds. Using the similar argument as above, we can get $\B(p,q+1)$.

	Hence by the inductive method, $\A(p,q)$ and $\B(p,q)$ hold for $p,q\geq 0$ and $p+q>0$. The proof is completed.
\end{proof}

\begin{proof}[Proof of Theorem \ref{Thm 1}]
	For $f\in\mathfrak{H}_{\mathbb{C}}^{\odot p}\otimes\mathfrak{H}_{\mathbb{C}}^{\odot q}$ with the expansion given by \eqref{complex kernel}, by the linearity of mappings $\I_{p,q}\left(\cdot \right) $ and $I_p\left(\cdot \right) $, we have	
	\begin{align}
		\I_{p,q}(f)&=\sum_{k_1,\dots,k_p,j_1,\dots,j_q=1}^{\infty}b_{p,q}(\k,\j)\I_{p,q}\left(\k,\j\right) \\
		&=\sum_{k_1,\dots,k_p,j_1,\dots,j_q=1}^{\infty}b_{p,q}(\k,\j) I_{p+q}\left(u_{p,q}\left(\k,\j\right)+\i v_{p,q}\left(\k,\j\right)\right) \\
		&=\sum_{k_1,\dots,k_p,j_1,\dots,j_q=1}^{\infty}I_{p+q}\left( b_{p,q}(\k,\j)\left( u_{p,q}\left(\k,\j\right)+\i v_{p,q}\left(\k,\j\right)\right)\right) .
	\end{align}
\end{proof}

\subsection{Proof of Theorem \ref{inverse First proof}}\label{Section5.4}
\begin{proof}[Proof of Theorem \ref{inverse First proof}]
	
	\textbf{Step 1} For $ m+n=p$, we show the unique representation for $$I_p\left( u_{1,0}(1)^{\otimes m}\otimes v_{1,0}(1)^{\otimes n}\right).$$
	Combining the definition of real Wiener-It\^o integral \eqref{def of real integral} and the relation between real and complex Hermite polynomials \eqref{inverse real and complex Hermite}, we have
	\begin{align}\label{inverse repre for sub basis}
		&I_p\left( u_{1,0}(1)^{\otimes m}\otimes v_{1,0}(1)^{\otimes n}\right) \\=&I_m\left( u_{1,0}(1)^{\otimes m}\right) I_n\left(  v_{1,0}(1)^{\otimes n}\right) \\
		=&H_m\left(W\left(  u_{1,0}(1)\right) \right) H_n\left( W\left(  v_{1,0}(1)\right) \right)\\
		=&\frac{\i ^n}{2^p}\sum_{j=0}^{p}\sum_{r+s=j}\binom{m}{r}\binom{n}{s}(-1)^sJ_{j, p-j}\left(Z\left( e_1\right) \right)\\
		=&\frac{\i ^n}{2^p}\sum_{j=0}^{p}\sum_{r+s=j}\binom{m}{r}\binom{n}{s}(-1)^s\I_{j, p-j}\left(e_1^{\otimes j}\otimes\overline{e}_1^{\otimes p-j}\right).
	\end{align}
	
	\textbf{Step 2} For	$\textbf{m}=\left\lbrace m_k \right\rbrace_{k=1}^{\infty}, \textbf{n}=\left\lbrace n_k \right\rbrace_{k=1}^{\infty} \in\Lambda$ with $|\textbf{m}|+|\textbf{n}|=p$, we give the representation for 
	$$	I_{p}\left(\mathrm{symm}\left( \otimes_{k=1}^{\infty}\left( u_{1,0}(k)^{\otimes m_k}  \otimes v_{1,0}(k)^{\otimes n_k}\right)\right)  \right). $$
	Let $\textbf{p}=\left\lbrace p_k\right\rbrace _{k=1}^{\infty}:=\textbf{m}+\textbf{n}$. By the definitions of real and complex Wiener-It\^o integrals, \eqref{def of real integral} and \eqref{def of complex integral}, respectively, we have
	\begin{align}
		&I_{p}\left(\mathrm{symm}\left( \otimes_{k=1}^{\infty}\left( u_{1,0}(k)^{\otimes m_k}  \otimes v_{1,0}(k)^{\otimes n_k}\right)\right)  \right)\\
		=& \prod_{k=1}^{\infty}I_{p_k}\left( u_{1,0}(k)^{\otimes m_k}  \otimes v_{1,0}(k)^{\otimes n_k}\right) \\
		\overset{\eqref{inverse repre for sub basis}}{=}&\prod_{k=1}^{\infty}\left( \sum_{j=0}^{p_k}\tilde{a}_{k,j}\I_{j, p_k-j}\left(e_k^{\otimes j}\otimes\overline{e}_k^{\otimes p_k-j}\right)\right) \\
		=&\sum_{\textbf{j}\leq \textbf{p}}\left( \prod_{i=1}^{\infty}\tilde{a}_{i,j_i}\right) \I_{\sum_{k=1}^{\infty}j_k,p-\sum_{k=1}^{\infty}j_k}\left(\otimes_{k=1}^{\infty} \left( e_k^{\otimes j_k}\otimes \overline{e}_k^{\otimes (p_k-j_k)} \right) \right) \\
		=&\sum_{l=0}^{p}\I_{l,p-l}\left(\sum_{ \textbf{j}\leq \textbf{p},|\textbf{j}|=l}\left( \prod_{i=1}^{\infty}\tilde{a}_{i,j_i}\right)\otimes_{k=1}^{\infty} \left( e_k^{\otimes j_k}\otimes \overline{e}_k^{\otimes (p_k-j_k)} \right) \right) .
	\end{align}
	
	\textbf{Step 3}	
	For $g_1, g_2\in\left( \mathfrak{H}\oplus\mathfrak{H}\right) ^{\odot p}$  and $g_1+\i g_2$ is given by \eqref{expansion for real},
	\begin{align}
		I_{p}\left(g_1+\i g_2\right)
		&=\sum_{\textbf{m},\textbf{n}\in\Lambda,|\textbf{m}|+|\textbf{n}|=p} \tilde{c}(\textbf{m};\textbf{n})I_{p}\left( \mathrm{symm}\left( \otimes_{k=1}^{\infty}\left( u_{1,0}(k)^{\otimes m_k}  \otimes v_{1,0}(k)^{\otimes n_k}\right)\right)\right) \\
		&=\sum_{\textbf{m},\textbf{n}\in\Lambda,|\textbf{m}|+|\textbf{n}|=p} \tilde{c}(\textbf{m};\textbf{n})\sum_{l=0}^{p}\I_{l,p-l}\left( g_{l,p-l}(\textbf{m},\textbf{n})\right) \\
		&= \sum_{l=0}^{p}\I_{l,p-l}\left(\sum_{\textbf{m},\textbf{n}\in\Lambda,|\textbf{m}|+|\textbf{n}|=p} \tilde{c}(\textbf{m};\textbf{n}) g_{l,p-l}(\textbf{m},\textbf{n})\right) .
	\end{align}
\end{proof}

\bibliographystyle{abbrv}
\bibliography{mybib}

\end{document}